\documentclass[11pt]{amsart}

\usepackage{amsfonts}
\usepackage{amscd}
\usepackage{amsmath,amssymb,amsthm,stmaryrd}
\usepackage{mathrsfs}
\usepackage{enumerate}
\usepackage{float} 
\usepackage[pdftex]{graphicx}
\allowdisplaybreaks

\usepackage{hyperref}

\usepackage{color}

\newtheorem{thm}{Theorem}[section]
\newtheorem{cor}[thm]{Corollary}
\newtheorem{lem}[thm]{Lemma}

\newtheorem{prop}[thm]{Proposition}
\theoremstyle{definition}

\theoremstyle{remark}
\newtheorem{rem}[thm]{Remark}
\numberwithin{equation}{section}

\newcommand{\R}{\mathbb R}
\newcommand{\Z}{\mathbb Z}

\newcommand{\C}{{\mathbb C}}

\renewcommand{\H}{\mathbb H}

\newcommand{\be}{\begin{equation}}
\newcommand{\ee}{\end{equation}}
\newcommand{\supp}{\operatorname{supp}}
\renewcommand{\Re}{\operatorname{Re}}
\renewcommand{\Im}{\operatorname{Im}}

\title[ Spherical means on the Heisenberg group]
{On the maximal function associated to the \\ 
 spherical means on the Heisenberg group}

\author[S. Bagchi, S. Hait, L. Roncal and S. Thangavelu]
{Sayan Bagchi \and Sourav Hait \and Luz Roncal \and Sundaram Thangavelu}

\address[S. Bagchi]{Stat-Math Unit, Indian Statistical Institute, Kolkata, India.}
\curraddr{Department of Mathematics and Statistics\\
 Indian Institute of Science Education and Research Kolkata\\
Mohanpur 741246, Nadia, West Bengal, India} 
\email{sayansamrat@gmail.com}

\address[S. Hait]{Department of Mathematics\\
 Indian Institute of Science\\
560 012 Bangalore, India} 
\email{souravhait@iisc.ac.in}

\address[L. Roncal]{BCAM - Basque Center for Applied Mathematics \\
48009 Bilbao, Spain and Ikerbasque, Basque Foundation for Science, 48011 Bilbao, Spain}
\email{lroncal@bcamath.org}

\address[S. Thangavelu]{Department of Mathematics\\
 Indian Institute of Science\\
560 012 Bangalore, India\\
and BCAM - Basque Center for Applied Mathematics \\
48009 Bilbao, Spain}
\email{veluma@iisc.ac.in}

\keywords{Spherical means, Heisenberg group, $L^p$-improving estimates, sparse domination, weighted theory.}
\subjclass[2010]{Primary: 43A80. Secondary: 22E25, 22E30, 42B15, 42B25.}

\dedicatory{Dedicated to the memory of Eli Stein} 

\begin{document}

\maketitle

\begin{abstract}
In this paper we deal with lacunary and full versions of the spherical maximal function on the Heisenberg group $\H^n$, for $n\ge 2$. 
By suitable adaptation of an approach developed by M. Lacey in the Euclidean case, we obtain sparse bounds for these maximal functions, which lead to new unweighted and weighted estimates. In particular, we deduce the $L^p$ boundedness, for $1<p<\infty$, of the lacunary maximal function associated to the spherical means on the Heisenberg group. In order to prove the sparse bounds, we establish $ L^p-L^q $ estimates for local (single scale) variants of the spherical means.  
\end{abstract}

\tableofcontents


\section{Introduction and main results}

A celebrated theorem of Stein \cite{Stein} proved in 1976 says that the spherical maximal function $ M $  defined by
$$M^{\operatorname{Euc}}_{\operatorname{full}}f(x) = \sup_{r>0} |f\ast \sigma_r(x)| =  \sup_{r>0}\Big| \int_{|y|=r} f(x-y) d\sigma_r(y)\Big|$$
is bounded on $ L^p(\R^n)$, $n \ge 3$, if and only if $ p > n/(n-1).$ Here $ \sigma_r $ stands for the normalised surface measure on the sphere $ S_r = \{ x\in \R^n: |x|=r\} $ in $ \R^n.$ The case $ n =2 $ was proved later by Bourgain \cite{Bourgain}.  As opposed to this,  in 1979, C. P. Calder\'on  \cite{C} proved that the lacunary spherical maximal function 
$$ 
M_{\text{lac}}^{\operatorname{Euc}}f(x) := \sup_{ j \in \Z} \Big|\int_{|y|=2^j} f(x-y) d\sigma_{2^j}(y)\Big|
$$
is bounded on $ L^p(\R^n) $ for all $ 1 < p < \infty $ for $ n \ge 2$. 

In a recent article, Lacey \cite{Lacey} revisited the spherical maximal function. Using a new approach, he managed to prove certain sparse bounds for these maximal functions which led him to obtain new weighted norm inequalities. One of the goals in this paper is to adapt the method of Lacey to obtain sparse bounds for certain spherical means on the Heisenberg group. As consequences, unweighted and weighted analogues of Calder\'on's theorem follow in this context. Up to our knowledge, these results are new.

Let $\H^n=\C^n\times \R$ be the $(2n+1)$-dimensional Heisenberg group with the group law
$$
(z,t)(w,s)=\Big(z+w,t+s+\frac12\Im z\cdot \overline{w}\Big). 
$$
Given a function $f$ on $\H^n$, consider the spherical means
\begin{equation}
\label{eq:defin}
A_rf(z,t):=f\ast \mu_r(z,t)=\int_{|w|=r}f\Big(z-w,t-\frac12\Im z\cdot \overline{w}\Big)\,d\mu_r(w)
\end{equation}
where $\mu_r$ is the normalised surface measure on the sphere $S_r=\{(z,0):|z|=r\}$ in $\H^n$. 
The maximal function associated to these spherical means was first studied by Nevo and Thangavelu in \cite{NeT}. Later, improving the results in \cite{NeT}, Narayanan and Thangavelu \cite{NaT}, and M\"uller and Seeger \cite{MS}, independently, proved the following sharp maximal theorem: the full maximal function 
$$ 
M_{\operatorname{full}}f(z,t) := \sup_{r>0} |A_rf(z,t)| 
$$ 
is bounded on $ L^p(\H^n), n \geq 2 $ if and only if $ p > (2n)/(2n-1).$

In this work we first consider the lacunary maximal function associated to the spherical means
$$
M_{\operatorname{lac}}f(z,t):=\sup_{j\in \Z}|A_{2^j}f(z,t)|, 
$$
 and prove the following result.
\begin{thm} 
\label{thm:spherical}
Assume that $ n \geq 2.$ Then the associated lacunary maximal funcion $ M_{\operatorname{lac}}$ is bounded on $ L^p(\H^n) $ for any $ 1 < p < \infty$.

\end{thm}


 We remark that another kind of spherical maximal function on the Heisenberg group has been considered by Cowling. In \cite{Cowling} he studied the maximal function associated to the spherical means taken over genuine Heisenberg spheres, i.e., averages over spheres defined in terms of a homogeneous norm on $ \H^n$,  and proved that it is bounded on $L^p(\H^n) $ for $ p > \frac{Q-1}{Q-2} $, where $ Q = (2n+2) $ is the homogeneous dimension of $ \H^n$. Recently, in \cite{GT}, lacunary maximal functions associated with these spherical means have been studied and it has been shown that they are bounded on $ L^p(\H^n)$ for all $ p > 1$. We remark in passing that the spherical means \eqref{eq:defin} are more singular, being supported on codimension two submanifolds as opposed to the one studied in \cite{Cowling}, which are supported on codimension one submanifolds. Even more singular spherical means have been studied in the literature, see e.g. \cite{TK}.

Theorem \ref{thm:spherical}, as well as certain weighted versions that are stated in Section \ref{sec:bound}, are standard consequences of the sparse bound in Theorem \ref{thm:mainSH}. Before stating the result let us set up the notation. As in the case of $ \R^n $, there is a notion of dyadic grids on $ \H^n $, the members of which are called (dyadic) cubes.
A collection of cubes $\mathcal{S}$ in $\H^n $ is said to be $ \eta$-sparse if there are sets $\{E_S \subset S:S\in \mathcal{S}\}$ which are pairwise disjoint and satisfy $|E_S|>\eta|S|$ for all $S\in \mathcal{S}$. For any cube $Q$ and $1<p<\infty$, we define
 \[ \langle f\rangle_{Q,p}:=\bigg(\frac{1}{|Q|}\int_Q|f(x)|^pdx\bigg)^{1/p}, \qquad  \langle f\rangle_{Q}:=\frac{1}{|Q|}\int_Q|f(x)|dx. \]
 In the above, $ x =(z,t) \in \H^n $ and $ dx = dz \,dt $ is the Lebesgue measure on $ \C^n \times \R $, which incidentally is the Haar measure on the Heisenberg group.
By the term  $(p,q)$-sparse form we mean the following: 
\[
\Lambda_{\mathcal{S},p,q}(f_1,f_2)=\sum_{S\in\mathcal{S}}|S|\langle f_1\rangle_{S,p}\langle f_2\rangle_{S,q}.
\]

\begin{thm}
\label{thm:mainSH}
Assume $ n \geq 2$. Let $ 1 < p, q < \infty $ be such that $ (\frac{1}{p},\frac{1}{q}) $ belongs to the interior of the triangle joining the points $ (0,1), (1,0) $ and $ (\frac{3n+1}{3n+4},\frac{3n+1}{3n+4})$. Then for any  pair of compactly supported bounded functions $ (f_1,f_2) $ there exists a $ (p,q)$-sparse form such that $ \langle M_{\operatorname{lac}}f_1,f_2\rangle \leq C \Lambda_{\mathcal{S},p,q}(f_1,f_2)$.
\end{thm}

We do not know whether Theorem \ref{thm:mainSH} delivers the optimal range of $(p,q)$. 
We will return to the study of the sharpness somewhere else. 

With a similar procedure, and using the results obtained for the lacunary case, we can also prove a sparse domination for the full maximal operator and deduce weighted norm inequalities, see Theorem \ref{thm:sparseF} in Subsection \ref{sec:sparseF}. Nevertheless, since these results are subordinated to the results for $M_{\operatorname{lac}}$, the bounds obtained are expected to be far from optimal. Indeed, as in the Euclidean case, the bounds are expected to hold in a quadrangle, rather than in a triangle, and better estimates along the anti-diagonal should be achieved.

In proving the corresponding sparse bounds for the spherical maximal functions on $ \R^n $, Lacey \cite{Lacey} made use of two features of the spherical means. The first one is the $ L^p-L^q$ estimate, also referred as $L^p$ improving estimate, of the operator $ S_r f = f \ast \sigma_r $ for a fixed $ r $, in the case of the lacunary spherical averages, and for a local (single scale) variant of the maximal function, in the case of the full averages. The second feature is a continuity  property of the difference $ S_r f- \tau_y S_r f $, where $ \tau_y f(x) = f(x-y) $ is the translation operator.  By this we mean a rescaled version of an estimate of the form $ \| S_1-\tau_y S_1 \|_{L^p \rightarrow L^q} \leq C |y|^\eta $ for some $ \eta >0.$ Thus this  is essentially a slight improvement of the $ L^p-L^q$ estimate, which turns out to be preserved under small translations, with a gain in $y$. In the Euclidean case, the $ L^p $ improving property of $S_r$ already existed in the literature, and the continuity property could be deduced almost immediately from the well-known estimates for the Fourier multiplier associated to these spherical means and the $L^p$ improving property. 

In our case, $L^p$ improving estimates, which are the heart of the matter, are new and addressed in Section~\ref{sec:Lp} for $A_r$. Our approach to develop the program and get the $ L^p-L^q$ estimates is based on spectral methods attached to the spherical means on the Heisenberg group. The continuity condition, even though it is a technical estimate that follows from the $ L^p-L^q$ bounds, is more difficult to obtain than in the Euclidean case and it is shown in Section \ref{sec:continuity}. The corresponding results concerning the full case are addressed in Section \ref{sec:full}. 

\begin{rem}
As mentioned above, we do not know whether our results are optimal or not  but actually we believe that they are  most probably suboptimal. In particular, for the full spherical maximal function, it is reasonable to expect the bounds to hold for a range of $(p,q)$ contained in a larger quadrangle, analogously as in the Euclidean case. Nevertheless, as it will be clear from the proofs, the procedure to obtain sparse bounds is independent of the numerology, so the suboptimality of the results are due to the suboptimal $L^p-L^q$ bounds for the single scale operators. The better input $L^p-L^q$ estimates would yield better sparse bounds.
\end{rem}

The results in this paper are restricted to dimension $n\ge2$. Recently, in \cite{BGHS},  the authors proved that $M_{\operatorname{full}}$, acting on a class of Heisenberg radial functions (i.e., $f:\H^1\to \C$ such that $f(Rz,t)=f(z,t)$ for all $R\in \operatorname{SO}(2)$), is bounded on $L^p(\H^1)$ for $2<p\le \infty$. Up to our knowledge, the boundedness of the full spherical maximal function on the Heisenberg group in the case $n=1$ is still open. 

\textbf{Outline of the proofs.}  We closely follow the strategy of Lacey  in proving Theorem \ref{thm:mainSH}, but in our case we do not have all the necessary ingredients at our disposal. Consequently, we have to first prove the $ L^p-L^q$ estimates of the operator $A_r$ on $\H^n$ and then  use them to prove the corresponding continuity property of the difference $ A_1f-A_1 \tau_y f$ where now $ \tau_y f(x) = f(xy^{-1}) $ is the right translation by $ y^{-1} $ on the Heisenberg group. We observe that, in the case of the Heisenberg group, the Fourier multipliers are not scalars but operators and hence the proofs become more involved. These results are new and have their own interest. Finally, we will prove the sparse bounds. We will have to modify appropriately the approach of Lacey, since we are in a non-commutative setting. This implies, in particular, that a metric has to be suitably chosen to make the Heisenberg group a space of homogeneous type. In order to keep the paper self-contained, we present a full detailed proof of the sparse domination. Along the paper, we will be assuming that the functions $f, f_1$, and $f_2$ arising are non-negative, which we can always do without loss of generality.

\textbf{Structure of the paper.} In Section \ref{sec:Lp} we give definitions and facts concerning the group Fourier transform on $\H^n$, the spectral description of the spherical means $A_r$, and  we establish $ L^p-L^q $ estimates for these operators. 
 In Section \ref{sec:continuity} we prove the continuity property  of $ A_r f- A_r\tau_y f.$ In Section \ref{sec:sparse} we establish the sparse bound and prove Theorem \ref{thm:mainSH} and in Section \ref{sec:bound} we deduce unweighted (Theorem \ref{thm:spherical}) and weighted boundedness properties of the lacunary maximal function. Finally, Section \ref{sec:full} is devoted to present the results for the full maximal function.
 
\section{$ L^p-L^q $ estimates for the spherical means}
\label{sec:Lp}

The observation that the spherical mean value operator $ S_r f := f\ast \sigma_r $ on $ \R^n $ is a Fourier multiplier plays an important role in every work dealing with the spherical maximal function. In fact, we know that
\begin{equation}
\label{eq:ArEuc}
 f \ast \sigma_r(x) = (2\pi)^{-n/2} \int_{\R^n}  e^{i x \cdot \xi} \widehat{f}(\xi) \frac{ J_{n/2-1}(r|\xi|)} {(r|\xi|)^{n/2-1}} d\xi 
\end{equation}
where $ J_{n/2-1} $ is the Bessel function of order $ n/2-1.$ As Bessel functions $ J_\alpha $ are defined even for complex values of $ \alpha $, the above allows one to embed $ S_r f$ into an analytic family of operators and Stein's analytic interpolation theorem comes in handy in studying the spherical maximal function. Indeed, this was the technique employed by Strichartz \cite{S} in order to study $ L^p $ improving properties of $ S_r $.  We will use the same strategy to get the $L^p$ improving property of $A_r$ on $\H^n$. Actually, for the spherical means on the Heisenberg group, there is available in the literature a representation  analogous to \eqref{eq:ArEuc} if we replace the Euclidean  Fourier transform by the group Fourier transform on $ \H^n$, see \eqref{eq:expression}.

The present section will be organised as follows. In Subsection \ref{subsec:preliminaries}  we will introduce some preliminaries on the group Fourier transform on $\H^n$. In Subsection  \ref{subsec:spectral} we will give the spectral description of the spherical averages $A_r$, which will involve special Hermite and Laguerre expansions. Sharp estimates for certain Laguerre functions will be shown in Subsection \ref{subsec:spectralest}. Then in Subsection~\ref{subsec:gFt} we will obtain the $L^p$ improving property of $A_r$. 

\subsection{The group Fourier transform on the Heisenberg group}
\label{subsec:preliminaries}

For the group $ \H^n $ we have a family  of irreducible unitary representations $ \pi_\lambda $ indexed by non-zero reals $ \lambda $ and realised on $ L^2(\R^n)$. The action of $\pi_{\lambda}(z,t)$ on $L^2(\R^n)$ is explicitly given by
\begin{equation}
\label{eq:pilambda}
\pi_{\lambda}(z,t)\varphi(\xi)=e^{i\lambda t}e^{i\lambda(x\cdot \xi+\frac12 x\cdot y)}\varphi(\xi+y)
\end{equation}
where $\varphi\in L^2(\R^n)$ and $z=x+iy$. 
By the theorem of Stone and von Neumann, which classifies all the irreducible unitary representations of $ \H^n $, combined with the fact that the Plancherel measure for $ \H^n $ is supported only on the infinite dimensional representations, it is enough to consider the following operator valued function known as the group Fourier transform of a given function $ f $ on $ \H^n$:
\begin{equation}
\label{eq:gFt}
\widehat{f}(\lambda) = \int_{\H^n} f(z,t) \pi_\lambda(z,t) \,dz \,dt.
\end{equation}
The above is well defined, e.g., when $ f \in L^1(\H^n) $ and for each $ \lambda \neq 0,$ $ \widehat{f}(\lambda) $ is a bounded linear operator on $ L^2(\R^n)$. The irreducible unitary representations $ \pi_\lambda $ admit the factorisation $ \pi_\lambda(z,t) = e^{i\lambda t} \pi_\lambda(z,0) $ and hence we can write the Fourier transform as
$$ \widehat{f}(\lambda)  = \int_{\C^n}  f^\lambda(z) \pi_\lambda(z,0)\, dz, $$ where
 for a function $f$ on $\H^n$, $f^{\lambda}(z)$ stands for the partial inverse Fourier transform
$$
f^{\lambda}(z)=\int_{-\infty}^{\infty}e^{i\lambda t}f(z,t)\,dt. 
$$

When $f\in L^1\cap L^2(\H^n)$ it can be easily verified that $\widehat{f}(\lambda)$ is a Hilbert-Schmidt operator and we have 
$$
\int_{\H^n}|f(z,t)|^2\,dz\,dt=(2\pi)^{-n-1}\int_{-\infty}^{\infty}\|\widehat{f}(\lambda)\|^2_{\operatorname{HS}}|\lambda|^n\,d\lambda.
$$
The above equality of norms allows us to extend the definition of the Fourier transform to all $L^2$ functions. It then follows that we have Plancherel theorem: $f\to \widehat{f}$ is a unitary operator from $L^2(\H^n)$ onto $L^2(\R^*,\textrm{S}_2,d\mu)$ where $\textrm{S}_2$ stands for the space of all Hilbert-Schmidt operators on $L^2(\R^n)$ and $d\mu(x)=(2\pi)^{-n-1}|\lambda|^{n}\,d\lambda$ is the Plancherel measure for the group $\H^n$. We refer to \cite{STH} for more details.
\subsection{Spectral theory of the spherical means on the Heisenberg group}
\label{subsec:spectral}

As pointed out above, a spectral definition of $A_r=f\ast \mu_r$ was already given in \cite{NaT, NeT}. For the convenience of the readers we will briefly recall it in this subsection after providing some necessary definitions that will be useful in the next sections.

Observe that the definition \eqref{eq:gFt} makes sense even if we replace $ f $ by a finite Borel measure $ \mu.$ In particular, $ \widehat{\mu}_r(\lambda) $ are well defined bounded operators on $ L^2(\R^n)$ which can be described explicitly. Combined with the fact that $ \widehat{ f\ast g}(\lambda) = \widehat{f}(\lambda) \widehat{g}(\lambda) $ we obtain $ \widehat{A_rf}(\lambda) = \widehat{f}(\lambda)\widehat{\mu}_r(\lambda).$
The operators $ \widehat{\mu}_r(\lambda) $ turn out to be diagonalisable in the Hermite basis. Indeed, if we let $ \Phi_\alpha^\lambda$,  $\alpha \in \mathbb{N}^n $, stand for the normalised  Hermite functions on $ \R^n, $ then $ \widehat{\mu}_r(\lambda) \Phi_\alpha^\lambda = \psi_k^{n-1}(\sqrt{|\lambda|}r) \Phi_\alpha^\lambda$ where $ k = |\alpha|$. Here, for any $ \delta > -1$, $\psi_k^{\delta}$ stand for the normalised Laguerre functions defined by 
\begin{equation}
\label{eq:LagF} 
\psi_k^{\delta}(r)=\frac{\Gamma(k+1)\Gamma(\delta+1)}{\Gamma(k+\delta+1)}L_k^{\delta}\Big(\frac12r^2\Big)e^{-\frac14r^2},
\end{equation}
where $  L_k^\delta(r) $ are the Laguerre polynomials of type $ \delta$.
The Hermite functions  $ \Phi_\alpha^\lambda $ are eigenfunctions of the Hermite operator $H(\lambda) = -\Delta+\lambda^2 |x|^2$. More precisely, $ H(\lambda) \Phi_\alpha^\lambda = (2|\alpha|+n)|\lambda| $ and the spectral decomposition of $ H(\lambda) $ is then written as 
\begin{equation}
\label{eq:Hl}
H(\lambda) = \sum_{k=0}^\infty  (2k+n)|\lambda| P_k(\lambda)
\end{equation} 
where $ P_k(\lambda) $ are the Hermite projection operators. It is well known (see \cite[Proposition 4.1]{TRMI}) that
$$ \widehat{\mu}_r(\lambda) = \sum_{k=0}^\infty  \psi_k^{n-1}(\sqrt{|\lambda|}r) P_k(\lambda), $$
 Hence we have the relation
\begin{equation}
\label{ArHe}
\widehat{A_rf}(\lambda) = \widehat{f}(\lambda) \sum_{k=0}^\infty  \psi_k^{n-1}(\sqrt{|\lambda|}r) P_k(\lambda),
\end{equation}
which is the analogue of \eqref{eq:ArEuc} in our situation. Thus, as in the Euclidean case, the spherical mean value operators $ A_r $ are (right) Fourier multipliers on the Heisenberg group. 

However, in order to define an analytic family of operators containing the spherical means, it is more suitable to rewrite \eqref{ArHe} in terms of Laguerre expansions. For that purpose, we will make use of the special Hermite expansion of the function $ f^\lambda $, which can be put in a compact form as follows.
Let $ \varphi_k^{\lambda}(z)=L_k^{n-1}\big(\frac12|\lambda||z|^2\big)e^{-\frac14|\lambda||z|^2}$ stand for the Laguerre functions of type $ (n-1) $ on $ \C^n.$ 
The $\lambda$-twisted convolution $f^{\lambda}\ast_{\lambda}\varphi_k^{\lambda}(z)$ is then defined by
$$
f^{\lambda}\ast_{\lambda}\varphi_k^{\lambda}(z)=\int_{\C^n}f^{\lambda}(z-w)\varphi_k^{\lambda}(w)e^{i\frac{\lambda}{2}\Im z\cdot \overline{w}}\,dw.
$$
It is well known that one has the expansion (see \cite[Chapter 3, proof of Theorem 3.5.6]{STH})
$$ f^\lambda(z)  = (2\pi)^{-n} |\lambda|^n \sum_{k=0}^\infty   f^{\lambda}\ast_{\lambda}\varphi_k^{\lambda}(z),
$$ 
which leads to the formula  (see \cite[Theorem 2.1.1]{STH})
$$ f(z,t) = (2\pi)^{-n-1}  \int_{-\infty}^\infty  e^{-i\lambda t} \Big( \sum_{k=0}^\infty   f^{\lambda}\ast_{\lambda}\varphi_k^{\lambda}(z)\Big) |\lambda|^n d\lambda. 
$$
Applying this to  $f\ast \mu_r$ we have 
$$
f\ast \mu_r(z,t)=\frac{1}{2\pi}\int_{-\infty}^{\infty}e^{-i\lambda t}f^{\lambda}\ast_{\lambda}\mu_r(z) \,d\lambda
$$
 where we used the fact that $(f\ast \mu_r)^{\lambda}(z)=f^{\lambda}\ast_{\lambda}\mu_r(z).$  It has been also shown that (see \cite[Theorem 4.1]{TRMI} and \cite[Proof of Proposition 6.1]{NeT}) 
$$
f^{\lambda}\ast_{\lambda}\mu_r(z)=  (2\pi)^{-n} |\lambda|^n \sum_{k=0}^{\infty}\frac{k!(n-1)!}{(k+n-1)!}\varphi_k^{\lambda}(r)f^{\lambda}\ast_{\lambda}\varphi_k^{\lambda}(z),
$$
leading to the expansion (see \cite{NaT, NeT})
\begin{equation}
\label{eq:expression}
A_rf(z,t)=(2\pi)^{-n-1}\int_{-\infty}^{\infty}e^{-i\lambda t}\Big(\sum_{k=0}^{\infty}\psi_k^{n-1}(\sqrt{|\lambda|}r)f^{\lambda}\ast_{\lambda}\varphi_k^{\lambda}(z)\Big)|\lambda|^n\,d\lambda.
\end{equation}
By replacing $ \psi_k^{n-1} $ by $ \psi_k^\delta $ we get the  family of operators taking $ f $ into 
\begin{equation}
\label{eq:ardelta}
(2\pi)^{-n-1}\sum_{k=0}^{\infty}\int_{-\infty}^{\infty}e^{-i\lambda t}\psi_k^{\delta}(\sqrt{|\lambda|}r)f^{\lambda}\ast_{\lambda}\varphi_k^{\lambda}(z)|\lambda|^n\, d\lambda.
\end{equation}
We will consider these operators when studying the $ L^p-L^q $ estimates of the spherical mean value operator.

\subsection{An analytic family of operators}
\label{subsec:analytic}

The Laguerre functions $\psi_k^{\delta}$ can be defined for all values of $\delta>-1$, even for complex $\delta$ with $\Re \delta>-1$. We define
\begin{equation}
\label{eq:Abeta}
 \mathcal{A}_{\beta}f(z,t)=(2\pi)^{-n-1}\int_{-\infty}^{\infty}e^{-i\lambda t}\Big(\sum_{k=0}^{\infty}\psi_k^{\beta+n-1}(\sqrt{|\lambda|})f^{\lambda}\ast_{\lambda}\varphi_k^{\lambda}(z)\Big)|\lambda|^n\,d\lambda,
\end{equation}
for $\Re(\beta+n-1)>-1$. Note that for $\beta=0$ we recover $A_1$, thus $A_1=\mathcal{A}_0$.
We will use the following relation between Laguerre polynomials of different types in order to express $\mathcal{A}_{\beta}$ in terms of $A_1$ (see \cite[(2.19.2.2)]{PBM})
\begin{equation}
\label{eq:connection}
L_k^{\mu+\nu}(r)=\frac{\Gamma(k+\mu+\nu+1)}{\Gamma(\nu)\Gamma(k+\mu+1)}\int_0^1t^{\mu}(1-t)^{\nu-1}L_k^{\mu}(rt)\,dt,
\end{equation}
valid for $\Re \mu>-1$ and $\Re \nu>0$. We define, for $ s>0$,
\begin{equation}
\label{eq:Poisson}
P_sf(z,t)=\frac{1}{2\pi}\int_{-\infty}^{\infty}e^{-i\lambda t}e^{-\frac14|\lambda| s}f^{\lambda}(z)\,d\lambda
\end{equation}
to be the Poisson integral of $f$ in the $t$-variable. We see that, for $\Re \beta>0$, $\mathcal{A}_{\beta}$ is given by the following representation.
\begin{lem}
\label{lem:family}
Let $\Re \beta>0$. The operator $\mathcal{A}_{\beta}$ is given by the formula
\begin{equation}
\label{eq:Abetaa}
\mathcal{A}_{\beta}f(z,t)=2\frac{\Gamma(\beta+n)}{\Gamma(\beta)\Gamma(n)}\int_0^1s^{2n-1}(1-s^2)^{\beta-1}P_{1-s^2}f\ast \mu_{s}(z,t)\,ds. 
\end{equation}
\end{lem}
\begin{proof}
In view of \eqref{eq:Abeta}, it is enough to verify
\begin{multline*}
2\frac{\Gamma(\beta+n)}{\Gamma(\beta)\Gamma(n)}\int_0^1s^{2n-1}(1-s^2)^{\beta-1}P_{1-s^2}f\ast \mu_{s}(z,t)\,ds\\
=(2\pi)^{-n-1}\int_{-\infty}^{\infty}e^{-i\lambda t}\Big(\sum_{k=0}^{\infty}\psi_k^{\beta+n-1}(\sqrt{|\lambda|})f^{\lambda}\ast_{\lambda}\varphi_k^{\lambda}(z)\Big)|\lambda|^n\,d\lambda.
\end{multline*}
Note that the left hand side of the above equation is well defined only for $\Re \beta>0$ whereas the right hand side makes sense for all $\Re \beta>-n$. We can thus think of the right hand side as an analytic continuation of the left hand side. 
In view of \eqref{eq:Poisson},  the Fourier transform of the Poisson integral $P_{s}f$ in the $t$-variable can be written as 
$$
(P_{s}f)^{\lambda}(z)=e^{-\frac14|\lambda| s}f^{\lambda}(z).
$$
Then, by \eqref{eq:expression}  the spherical averages of the Poisson integral $P_{1-s^2}f $ are given by 
$$
P_{1-s^2}f\ast \mu_{s}(z,t)=(2\pi)^{-n-1}\sum_{k=0}^{\infty}\int_{-\infty}^{\infty}e^{-i\lambda t}\psi_k^{n-1}(\sqrt{|\lambda|}s)e^{-\frac14|\lambda|(1-s^2)}f^{\lambda}\ast_{\lambda}\varphi_k^{\lambda}(z)|\lambda|^n\,d\lambda.
$$
Integrating the above equation against $s^{2n-1}(1-s^2)^{\beta-1}\,ds$, we obtain
$$
\int_0^1s^{2n-1}(1-s^2)^{\beta-1}P_{1-s^2}f\ast \mu_{s}(z,t)\,ds=(2\pi)^{-n-1}\sum_{k=0}^{\infty}\int_{-\infty}^{\infty}e^{-i\lambda t}\rho_k(\sqrt{|\lambda|})f^{\lambda}\ast_{\lambda}\varphi_k^{\lambda}(z)|\lambda|^n\,d\lambda,
$$
where
\begin{equation}
\label{eq:rho}
\rho_k(\sqrt{|\lambda|})=\int_0^1s^{2n-1}(1-s^2)^{\beta-1}\psi_k^{n-1}(\sqrt{|\lambda|}s)e^{-\frac14|\lambda|(1-s^2)}\,ds.
\end{equation}
Recalling the definition of $\psi_k^{n-1}$ given in \eqref{eq:LagF} we have
$$
\rho_k(\sqrt{|\lambda|})=\frac{\Gamma(k+1)\Gamma(n)}{\Gamma(k+n)}\int_0^1s^{2n-1}(1-s^2)^{\beta-1}L_k^{n-1}\Big(\frac12s^2|\lambda|\Big)e^{-\frac14|\lambda|}\,ds.
$$
We now use the formula \eqref{eq:connection}. First we make a change of variables $t\to u^2$ and then choose $\mu=n-1$ and $\nu=\beta$, so that
\begin{align}
\label{eq:rho2}
\notag\rho_k(\sqrt{|\lambda|})&=\frac{\Gamma(k+1)\Gamma(n)}{\Gamma(k+n)}\frac12\frac{\Gamma(\beta)\Gamma(k+n)}{\Gamma(k+n+\beta)}e^{-\frac14|\lambda|}L_k^{n+\beta-1}\Big(\frac{|\lambda|}{2}\Big)\\
&=\frac12\frac{\Gamma(\beta)\Gamma(n)}{\Gamma(\beta+n)}\psi_k^{\beta+n-1}(\sqrt{|\lambda|}).
\end{align}
The proof follows readily.
\end{proof}
In particular, from \eqref{eq:rho} and \eqref{eq:rho2} (after a change of variable), in the proof of Lemma~\ref{lem:family}, we infer the following identity.
\begin{cor}
\label{cor:ident}
Let $\Re \beta>0$ and $\alpha>-1$. Then, for $ r>0$,
$$
\psi_k^{\alpha+\beta}(r)=2\frac{\Gamma(\beta+\alpha+1)}{\Gamma(\beta)\Gamma(\alpha+1)}\int_0^1u^{\alpha}(1-u)^{\beta-1}\psi_k^{\alpha}(r\sqrt{u})e^{-\frac14 r^2(1-u)}\,du. 
$$
\end{cor}

Observe that even for  large $ \beta $,   the operator $ \mathcal{A}_\beta $ is a convolution operator with a distribution supported on $ \C^n \times \{0\} $. This is in sharp contrast with the Euclidean case, see \cite{S}, and prevents us to have some $ L^p $ improving property for large values of $ \beta$. In order to overcome this, we slightly modify the family in Lemma \ref{lem:family} and define a new family $T_{\beta}$. As we will see below the modified family of operators $ T_\beta$  has a better behaviour for $ \beta \ge 1$. 

For $ \Re \beta>0, $ let
\begin{equation}
\label{eq:kbeta}
k_\beta(t)=\frac{1}{\Gamma(\beta)}t_+^{\beta-1}e^{-t},  
\end{equation}
where $ t_+^{\beta-1} = t^{\beta-1} \chi_{(0,\infty)}(t) $, which defines a family of distributions on $\R$ and $\lim_{\beta\to 0} k_\beta(t)=\delta_0$, the Dirac distribution at $0$. Given a function $f$ on $\H^n$ and $\varphi$ on $\R$ we use the notation $f\ast_3 \varphi$ to stand for the convolution in the central variable:
$$
f\ast_3 \varphi(z,t)=\int_{-\infty}^{\infty}f(z,t-u)\varphi(u)\,du.
$$
Thus we note that $P_{1-s^2}f(z,t)=f\ast_3 p_{1-s^2}(z,t)$ where $p_{1-s^2}$ is the usual Poisson kernel in the one dimensional variable $t$, associated to $P_{1-s^2}$. In fact, $p_{s}(t) $ is defined by the relation 
$ \int_{-\infty}^\infty  e^{i \lambda t} p_{s}(t) dt = e^{-\frac{1}{4}s|\lambda|}  $ and it is explicitly known: $ p_{s}(t) = c s(s^2+16 t^2)^{-1}$ for some constant $ c>0,$ see for example \cite{SW}.
With the above notation we define the new family by
\begin{equation}
\label{eq:Tbeta}
T_{\beta}f(z,t)=\frac{\Gamma(\beta+n)}{\Gamma(\beta)\Gamma(n)}\int_0^1s^{2n-1}(1-s^2)^{\beta-1}P_{1-s^2}(f\ast_3 k_\beta)\ast \mu_{s}(z,t)\,ds.
\end{equation}
In other words 
$$
T_{\beta}f= \mathcal{A}_{\beta}(f\ast_3 k_\beta).
$$ 
\begin{lem}
\label{lem:explicitTb}
For $ \Re \beta>0, $ the operator $T_{\beta}f$ has the explicit expansion
\begin{multline*}
T_{\beta}f(z,t)\\
=(2\pi)^{-n-1}\int_{-\infty}^{\infty}e^{-i\lambda t}(1-i\lambda)^{-\beta}\Big(\sum_{k=0}^{\infty}\psi_k^{\beta+n-1}(\sqrt{|\lambda|})f^{\lambda}\ast_{\lambda}\varphi_k^{\lambda}(z)\Big)|\lambda|^n\,d\lambda.
\end{multline*}
\end{lem}
\begin{proof}
The statement follows from Lemma \ref{lem:family}, \eqref{eq:Abeta}, and from the fact 
$$
\int_{-\infty}^{\infty}e^{i\lambda t}k_\beta(t)\,dt=\frac{1}{\Gamma(\beta)}\int_0^{\infty}e^{i\lambda t}t^{\beta-1}e^{-t}\,dt=(1-i\lambda)^{-\beta}.
$$
This can be verified by considering the function
$$
F(\beta, \zeta)=\frac{1}{\Gamma(\beta)}\int_0^{\infty}t^{\beta-1}e^{-t \zeta}\,dt
$$
defined and holomorphic for $\Re \beta>0$, $\Re \zeta>0$. Indeed, when $\zeta$ is fixed, with $\Re \zeta >0$, we have the relation $F(\beta,\zeta)=\zeta F(\beta+1,\zeta)$ which allows us to holomorphically extend $F(\beta,\zeta)$ in the $\beta$ variable. It is clear that when $\zeta>0$, $F(\beta,\zeta)=\zeta^{-\beta}$, which allows us to conclude that the Fourier transform of $ k_\beta$ at $\lambda$ is given by $(1-i\lambda)^{-\beta}$, as claimed.
\end{proof}

\subsection{Spectral estimates}
\label{subsec:spectralest}

In this subsection we will state and prove sharp estimates on the normalised Laguerre functions given in \eqref{eq:LagF}. These estimates will be needed to get the $L^2$ boundedness of the analytic family operators that we introduced in the previous subsection. 

We begin by expressing $\psi_k^{\delta}(r)$ more conveniently in terms of the standard Laguerre functions
$$
\mathcal{L}_k^{\delta}(r)=\Big(\frac{\Gamma(k+1)\Gamma(\delta+1)}{\Gamma(k+\delta+1)}\Big)^{\frac12}L_k^{\delta}(r)e^{-\frac12r}r^{\delta/2}
$$
which  form an orthonormal system in $L^2((0,\infty),dr)$. In terms of $\mathcal{L}_k^{\delta}(r),$ we have
$$
\psi_k^{\delta}(r)=2^{\delta}\Big(\frac{\Gamma(k+1)\Gamma(\delta+1)}{\Gamma(k+\delta+1)}\Big)^{\frac12}r^{-\delta}\mathcal{L}_k^{\delta}\Big(\frac12r^2\Big).
$$
Asymptotic properties of $\mathcal{L}_k^{\delta}(r)$ are well known in the literature and we have the following result, see \cite[Lemma 1.5.3]{T} (actually, the estimates in Lemma \ref{lem:T} below are sharp, see \cite[Section 2]{M} and \cite[Section 7]{Mu}).

\begin{lem}[\cite{T}]
\label{lem:T}
For $ k \in \mathbb{N} $, let us  set $ \nu: = (4k+2\delta+1).$ Then for $\delta>-1$, we have the following: 
$$
|\mathcal{L}_k^{\delta}(r)|\le C\begin{cases}(\nu r)^{\delta/2}, &0\le r\le \frac{1}{\nu}\\
(\nu r)^{-\frac14}, &\frac{1}{\nu}\le r\le \frac{\nu}{2}\\
\nu^{-\frac14}(\nu^{\frac13}+|\nu-r|)^{-\frac14}, &\frac{\nu}{2}\le r\le \frac{3\nu}{2}\\
e^{-\gamma r}, & r\ge \frac{3\nu}{2},
\end{cases}
$$
where $\gamma>0$ is a fixed constant.
\end{lem}
From the above estimates of $\mathcal{L}_k^{\delta}$ we can obtain the following estimates for  the normalised Laguerre functions $\psi_k^{\delta}$.

\begin{lem}
\label{lem:uniform}
For $k\in \mathbb{N}$, let us set $\nu:=(4k+2\delta+1)$. Then, for $\alpha\ge0$ and $\delta>-1$ such that $\delta+\frac13-2\alpha\ge 0$, we have the uniform estimates
$$
\sup_k(\nu|\lambda|)^{\alpha}|\psi_k^{\delta}(\sqrt{|\lambda|})|\le C\begin{cases}
1,  &\text{ if } |\lambda|\le  \frac{1}{\nu},\\
|\lambda|^{2\alpha-\delta-\frac13},  &\text{ if } |\lambda|> \frac{1}{\nu}.
\end{cases}
$$
\end{lem}
\begin{proof}
Since $\frac{\Gamma(k+1)\Gamma(\delta+1)}{\Gamma(k+\delta+1)}\le C (4k+2\delta+1)^{-\delta}$ we need to bound 
$$
(\nu|\lambda|)^{\alpha}(\nu |\lambda|)^{-\delta/2}\mathcal{L}_k^{\delta}\Big(\frac12|\lambda|\Big). 
$$
When $\frac{1}{\nu}\le \frac12|\lambda|\le \frac{\nu}{2}$ we have the estimate 
$$
(\nu|\lambda|)^{\alpha}|\psi_k^{\delta}(\sqrt{|\lambda|})|\le C(\nu |\lambda|)^{\alpha-\delta/2-1/4}.
$$
From here, since $-2\alpha+\delta+\frac12\ge 0$, $\lambda^2\le  \nu |\lambda|$, we get 
$$
(\nu|\lambda|)^{\alpha}|\psi_k^{\delta}(\sqrt{|\lambda|})|\le C|\lambda|^{2\alpha-\delta-1/2}.
$$
When $\frac{\nu}{2}\le \frac12|\lambda|\le \frac{3\nu}{2}$, $|\lambda|$ is comparable to $\nu$ and hence we have
$$
(\nu|\lambda|)^{\alpha}|\psi_k^{\delta}(\sqrt{|\lambda|})|\le C(\nu|\lambda|)^{\alpha}(\nu |\lambda|)^{-\delta/2}\nu^{-\frac14}\nu^{-\frac{1}{12}}\le C|\lambda|^{2\alpha-\delta-\frac13}.
$$
On the region $|\lambda|\ge \frac{3\nu }{2}$ we have exponential decay. Finally, the estimate $\sup_k(\nu|\lambda|)^{\alpha}|\psi_k^{\delta}(\sqrt{|\lambda|})|\le C$ for $0\le |\lambda|\le\frac{1}{\nu}$ is immediate, in view of Lemma \ref{lem:T}. With this we prove the lemma. 
\end{proof}


\subsection{$ L^p-L^q $ estimates}
\label{subsec:gFt}

After the preparations in the previous subsections, we will proceed to prove the $ L^p-L^q $ estimates of the operator $A_1$. 

We will show  that when $\beta=1+i\gamma$, the operator $T_{\beta}$ defined in \eqref{eq:Tbeta} is bounded from $L^p(\H^n)$ into $L^{\infty}(\H^n)$ for any $p>1$, and that for certain negative values of $\beta$, $T_{\beta}$ is bounded on $L^2(\H^n)$. We can then use analytic interpolation to obtain a result for $T_0=\mathcal{A}_0=A_1$. We shall use the following definition: A function $\Phi(z)$ analytic in the open strip $0<\Re(z)<1$, and continuous in the closed strip, will be called of \textit{admissible growth} (cf. \cite{Stein-I}) if 
$$
\sup_{|y|\le r}\sup_{0\le x\le 1}\log|\Phi(x+iy)|\le Ae^{ar}, \quad a<\pi.
$$
\begin{prop}
\label{prop:Linfty}
Assume that $n\ge1$. Then for any $\delta>0$, $\gamma\in \R$,
$$
\|T_{1+i\gamma}f\|_{\infty}\le C_1(\gamma)\|f\|_{1+\delta},
$$
where $C_1(\gamma)$ is of admissible growth.
\end{prop}
\begin{proof}
For $\beta=1+i\gamma$ it follows that 
$$
|T_{1+i\gamma}f(z,t)|\le \frac{|\Gamma(1+i\gamma+n)|}{|\Gamma(1+i\gamma)|^2\Gamma(n)}\int_0^1s^{2n-1}P_{1-s^2}(f\ast_3 \varphi)\ast \mu_{s}(z,t)\,ds
$$
where $\varphi(t)=e^{-t}\chi_{(0,\infty)}(t)$. Since $\varphi\ge0$ it follows that  
$$
P_{1-s^2}(f\ast_3 \varphi)=\varphi\ast_3p_{1-s^2}\ast_3f\le \varphi\ast_3M_{\operatorname{HL}}^0 f
$$
where $M_{\operatorname{HL}}^0 f$ is the Hardy-Littlewood maximal function in the $t$-variable (we will use the notation  $M_{\operatorname{HL}} f$ for the Hardy-Littlewood maximal functions in all the variables in Section \ref{sec:full}). Here we have used the following well known fact:  Let $ \psi $ be a non-negative, integrable and radially decreasing function on $ \R $ and set $ \psi_s(t) = s^{-1}\psi(t/s).$ Then   $\sup_{s>0}|g \ast \psi_{s}(t)|\le C M g(t)$ for any locally integrable function $g$ on $\R$  where $ Mg $ stands for the Hardy-Littlewood maximal function on $ \R.$ Thus we have the estimate
$$
|T_{1+i\gamma}f(z,t)|\le C_1(\gamma)\int_0^1 (M_{\operatorname{HL}}^0 f\ast_3 \varphi)\ast \mu_{s}(z,t)s^{2n-1}\,ds.
$$
Now we make the following observation: suppose $K(z,t)=k(|z|)\varphi(t)$, where $k$ is a non-negative function on $ [0,\infty).$ Then
$$
f\ast K(z,t)=\int_0^{\infty}(f\ast_3 \varphi)\ast \mu_{s}(z,t)k(s)s^{2n-1}\,ds,
$$
which can be verified by recalling the definition of the spherical means $f\ast\mu_{s}(z,t)$ in \eqref{eq:defin} and integrating in polar coordinates. This gives us 
$$
|T_{1+i\gamma}f(z,t)|\le C_1(\gamma)M_{\operatorname{HL}}^0 f\ast K(z,t)
$$
where $K(z,t)=\chi_{|z|\le 1}(z)\varphi(t)$. As $M_{\operatorname{HL}}^0 f\in L^{1+\delta}(\H^n)$ and $K\in L^q(\H^n)$ for any $q\ge 1$, by H\"older we get
$$
\|T_{1+i\gamma}f\|_{\infty}\le C_1(\gamma)\|M_{\operatorname{HL}}^0 f\|_{1+\delta}\le C_1(\gamma)\|f\|_{1+\delta}.
$$
\end{proof}

In the next proposition we show that $T_{\beta}$ is bounded on $L^2(\H^n)$ even for some values of $\beta<0$. 
\begin{prop}
\label{prop:L2}
Assume that $n\ge1$ and $\beta> -\frac{n}{2}+\frac13$. Then for any $\gamma\in \R$
$$
\|T_{\beta+i\gamma}f\|_{2}\le C_2(\gamma)\|f\|_{2}.
$$
\end{prop}
\begin{proof}
 In view of Lemma \ref{lem:explicitTb} and Plancherel theorem for the Fourier transform on $\R$  and  special Hermite expansions on $ \C^n$, we only have to check (observe that $ |(1-i\lambda)| = (1+\lambda^2)^{1/2}$),
$$
|(1-i\lambda)^{-(\beta+i\gamma)}||\psi_k^{\beta+i\gamma+n-1}(\sqrt{|\lambda|})|= (1+\lambda^2)^{-\beta/2}|\psi_k^{\beta+i\gamma+n-1}(\sqrt{|\lambda|})|\le C_2(\gamma)
$$
where $C_2(\gamma)$ is independent of $k$ and $\lambda$. When $\gamma=0$, it follows from the estimates of Lemma~\ref{lem:uniform} (with $\alpha=0$) that
$$
(1+\lambda^2)^{-\beta/2}|\psi_k^{\beta+n-1}(\sqrt{|\lambda|})|\le C|\lambda|^{-\beta}|\lambda|^{-\beta-(n-1)-\frac13}
$$
for $|\lambda|\ge 1$ (actually, for $|\lambda|\ge\frac{1}{\nu}$), which is bounded for  $\beta\ge -\frac{n}{2}+\frac13$. For $\gamma \neq 0$ we can express $\psi_k^{\beta+i\gamma+n-1}(\sqrt{|\lambda|})$ in terms of $\psi_k^{\beta-\varepsilon+n-1}(\sqrt{|\lambda|})$ for a small enough $\varepsilon>0$ and obtain the same estimate. Indeed, by Corollary~\ref{cor:ident} with $\alpha=\beta-\varepsilon+n-1$ and $\beta=\varepsilon+i\gamma$, and using the asymptotic formula $|\Gamma(\mu+iv)|\sim \sqrt{2\pi}|v|^{\mu-1/2}e^{-\pi|v|/2}$, as $v\to \infty$ (see for instance \cite[p. 281 bottom note]{SW}), we get
\begin{align*}
&|\psi_k^{\beta+i\gamma+n-1}(\sqrt{|\lambda|})|=\Big|2\frac{\Gamma(\beta+i\gamma+n)}{\Gamma(\varepsilon+i\gamma)\Gamma(\beta-\varepsilon+n)}\\
&\quad \times \int_0^1s^{\beta-\varepsilon+n-1}(1-s)^{\varepsilon+i\gamma-1}\psi_k^{\beta-\varepsilon+n-1}(\sqrt{|\lambda|s})e^{-\frac14|\lambda|(1-s)}\,ds\Big|\\
&\lesssim \frac{|\gamma|^{\beta+n-1/2}}{|\gamma|^{\varepsilon-1/2}}\Big|\int_0^1s^{\beta-\varepsilon+n-1}(1-s)^{\varepsilon+i\gamma-1}\psi_k^{\beta-\varepsilon+n-1}(\sqrt{|\lambda|s})e^{-\frac14|\lambda|(1-s)}\,ds\Big|,
\end{align*}
where the constant depends on $\beta$. Now, by the estimate for $\psi_k^{\delta}$ in Lemma \ref{lem:uniform} (for $\alpha=0$) and the integrability of the function $s^{\beta-\varepsilon+n-1}(1-s)^{\varepsilon+i\gamma-1}$, we have
$$
(1+\lambda^2)^{-\beta/2}|\psi_k^{\beta+i\gamma+n-1}(\sqrt{|\lambda|})|\le C|\lambda|^{-\beta}|\gamma|^{\beta+n-1-\varepsilon}|\lambda|^{-(\beta+n-1-\varepsilon)-1/3}.
$$
For $|\lambda|\ge 1$, the above is bounded for $\beta-\varepsilon\ge -\frac{n}{2}+\frac13$ with $\varepsilon$ small enough. The proof is complete.
\end{proof}

\begin{thm}
\label{thm:Lp}
Assume that $n\ge 1$ and $\varepsilon>0$. Then $A_1:L^{p}(\H^n)\to L^{q}(\H^n)$ for any $p,q$ such that 
$$
\frac{3}{3n+4-6\varepsilon}<\frac{1}{p}<\frac{3n+1-6\varepsilon}{3n+4-6\varepsilon},\qquad\frac{1}{q}=\frac{3}{3n+4-6\varepsilon}.
$$ 
\end{thm}
\begin{proof}
Let us consider the following holomorphic function $\alpha(z)$ on the strip $\{z:0\le \Re z\le 1\}$, given by  $\alpha(z)=\big(\frac{n}{2}-\frac13-\varepsilon\big)(z-1)+z$. We have $\alpha(0)=-\frac{n}{2}+\frac13+\varepsilon$ and $\alpha(1)=1$. Then, $T_{\alpha(z)}$ is an analytic family of linear operators and it was already shown that $T_{1+i\gamma}$ is bounded from $L^{1+\delta}(\H^n)$ to $L^{\infty}(\H^n)$. Therefore, we can apply Stein's interpolation theorem. Letting $z=u+iv$, we have
$$
\alpha(z)=0 \Longleftrightarrow \Big(\frac{n}{2}-\frac13-\varepsilon\Big)(u-1)+u=0 \Longleftrightarrow u=\frac{3n-2-6\varepsilon}{3n+4-6\varepsilon}.
$$
Since $\varepsilon>0$ is arbitrary,
we obtain
$$
T_{\alpha(u)}:L^{p_u}(\H^n)\to L^{q_u}(\H^n)
$$
where 
$$
\frac{3}{3n+4-6\varepsilon}<\frac{1}{p_u}<\frac{3n+1-6\varepsilon}{3n+4-6\varepsilon},\qquad\frac{1}{q_u}=\frac{3}{3n+4-6\varepsilon},
$$ 
and this leads to the result.
\end{proof}

\begin{cor}
\label{cor:LpLq}
Assume that $n\ge 1$. Then 
$$
A_1:L^{p}(\H^n)\to L^{q}(\H^n)
$$
whenever $\big(\frac{1}{p},\frac{1}{q}\big)$ lies in the interior of the triangle joining the points $(0,0), (1,1)$ and $\big(\frac{3n+1}{3n+4},\frac{3}{3n+4}\big)$, as well as the straight line segment joining the points $(0,0), (1,1)$, see $\mathbf{S}_n'$ in Figure \ref{figureaSH}.
\end{cor}
\begin{proof}
The result follows from Theorem \ref{thm:Lp} after applying Marcinkiewicz interpolation theorem with the obvious estimates $\|A_1f\|_1\le \|f\|_1$ and
$\|A_1f\|_{\infty}\le \|f\|_{\infty}$. 
\end{proof}

\begin{figure}[t]
\includegraphics[scale=1]{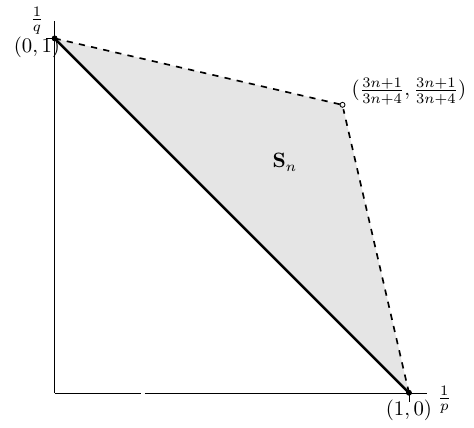}
\includegraphics[scale=1]{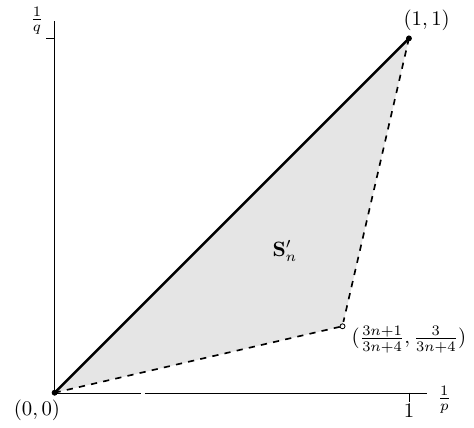}
\caption{Triangle $\mathbf{S}_n'$ shows the region for $L^p-L^q$ estimates for $A_1$. The dual triangle $\mathbf{S}_n$ is on the left.}
\label{figureaSH}
\end{figure}

\begin{rem}
\label{rem:restrno}
Observe that the results in this section are valid for dimensions $n\ge1$. The restriction $n\ge2$ will arise in Proposition \ref{prop:continuity} as a consequence of the restriction of the parameter $\delta$ on the available estimates for the Laguerre functions in Lemmas~\ref{lem:T} and \ref{lem:uniform}, which are sharp. Consequently, the rest of results from Proposition \ref{prop:continuity} on, and in particular the main results in this paper (Theorems \ref{thm:spherical} and \ref{thm:mainSH}), are restricted to dimensions $n=2$ and higher. 
\end{rem}

\section{The continuity property of the spherical means}
\label{sec:continuity}

In the work of Lacey \cite{Lacey} dealing with the spherical maximal function on $ \R^n $, the continuity property of the spherical mean value operator, described in the Introduction, played a crucial role in getting the sparse bounds for the spherical maximal function. It was obtained by combining the $ L^p-L^q $ estimates and $L^2$ estimates that were easily deduced from the known decay estimates of the Fourier multiplier associated to the spherical means.
 In the case of the Heisenberg group, the analogous property for $A_r$ is stated in Corollary~\ref{cor:dilat2} below. In order to achieve Corollary~\ref{cor:dilat2}, we will appeal to the $L^p$ improving estimates in Corollary \ref{cor:LpLq} along with suitable $L^2$ estimates. But in our setting, these $L^2$ estimates are not that immediate to obtain, since the associated multiplier  is an operator-valued function. This means that we are led to prove good  decay estimates on the norm of an  operator-valued function, which is  nontrivial.

In what follows, for $x=(z,t)\in \H^n$, we will denote by $ |x| = |(z,t)| = (|z|^4+t^2)^{1/4}$ the Koranyi norm on $ \H^n $.
\begin{prop}
\label{prop:continuity}
Assume that $ n \geq 2.$ Then for $ y \in \H^n, |y| \leq 1, $ we have
$$  \| A_1 - A_1 \tau_y \|_{L^2\to L^2} \leq C |y| $$
where $ \tau_y f(x) = f(xy^{-1}) $ is the right translation operator.
\end{prop}
\begin{proof}
For $ f \in L^2(\H^n) $ we estimate the $ L^2 $ norm of  $ A_1f -A_1(\tau_yf) $ using Plancherel theorem for the Fourier transform on $ \H^n.$ Recall that $ A_1f(x) = f \ast \mu_1(x) $ so that $ \widehat{A_1f}(\lambda) =  \widehat{f}(\lambda)\widehat{\mu_1}(\lambda)$, where $ \widehat{\mu_1}(\lambda) $ is explicitly given by
$$ \widehat{\mu_1}(\lambda) = \sum_{k=0}^\infty \psi_k^{n-1}(\sqrt{|\lambda|}) P_k(\lambda) .$$
We also have 
$$ \widehat{\tau_yf}(\lambda) = \int_{\H^n} f(xy^{-1}) \pi_\lambda(x) dx = \widehat{f}(\lambda)\pi_\lambda(y).$$
Thus by the Plancherel theorem for the Fourier transform we have
$$ \| A_1f -A_1(\tau_yf)\|_2^2 =  c_n \int_{-\infty}^\infty \| \widehat{f}(\lambda)(I-\pi_\lambda(y))\widehat{\mu_1}(\lambda)\|_{\operatorname{HS}}^2 |\lambda |^n d\lambda.$$
Since the space of all Hilbert-Schmidt operators is a two sided ideal in the space of all bounded linear operators, it is enough to estimate the operator norm of $ (I-\pi_\lambda(y))\widehat{\mu_1}(\lambda)$ (for more about Hilbert-Schmidt operators see \cite{VSS}). Again, $ \widehat{\mu_1}(\lambda) $ is self adjoint and $ \pi_\lambda(y)^\ast = \pi_\lambda(y^{-1})$ and so we will estimate $ \widehat{\mu_1}(\lambda) (I-\pi_\lambda(y)).$

We make use of the fact that for every $ \sigma \in U(n) $ there is a unitary operator $ \mu_\lambda(\sigma) $ acting on $ L^2(\R^n) $ such that $ \pi_\lambda(\sigma z, t) = \mu_\lambda(\sigma)^\ast \pi_\lambda(z,t) \mu_\lambda(\sigma) $ for all $ (z,t) \in \H^n$. Indeed, this follows from the well known Stone--von Neumann theorem which says that any irreducible unitary representation of the Heisenberg group which acts like $ e^{i\lambda t}I $ when restricted to the center is unitarily equivalent to $ \pi_\lambda$. Moreover, $ \mu_\lambda $ has an extension to a double cover of the symplectic group as a unitary representation and is called the metaplectic representation, see \cite[Chapter 4, Section 2]{F}.

Given $ y = (z, t) \in \H^n $ we can choose $ \sigma \in U(n) $ such that $ y =(|z|\sigma e_1,t) $ where $ e_1 = (1,0,....,0).$ Thus 
$$ \pi_\lambda(y) = \mu_\lambda(\sigma)^\ast \pi_\lambda(|z|e_1,t)\mu_\lambda(\sigma).$$ Also, it is well known that $ \mu_\lambda(\sigma) $ commutes with functions of the Hermite operator $H(\lambda)$ given in \eqref{eq:Hl}. Since $ \widehat{\mu_1}(\lambda) $ is a function of $ H(\lambda) $ it follows that
$$ \widehat{\mu_1}(\lambda) (I-\pi_\lambda(z,t)) =\mu_\lambda(\sigma)^\ast  \widehat{\mu_1}(\lambda) (I-\pi_\lambda(|z|e_1,t))\mu_\lambda(\sigma).$$ 
Thus it is enough to estimate the operator norm of $ \widehat{\mu_1}(\lambda) (I-\pi_\lambda(|z|e_1,t)).$ In view of the factorisation $ \pi_\lambda(|z|e_1,t) = \pi_\lambda(|z|e_1,0) \pi_\lambda(0,t) $ we have that
\begin{multline*}
I-\pi_\lambda(|z|e_1,t)\\
=I- \pi_\lambda(|z|e_1,0) \pi_\lambda(0,t)=(I- \pi_\lambda(0,t))+(I-\pi_\lambda(|z|e_1,0)) \pi_\lambda(0,t)
\end{multline*}
so it suffices to estimate the norms of $ \widehat{\mu_1}(\lambda) (I-\pi_\lambda(0,t)) $ and $\widehat{\mu_1}(\lambda) (I-\pi_\lambda(|z|e_1,0)) \pi_\lambda(0,t) $  separately. Moreover, we only have to estimate them for  $ |\lambda| \geq 1$  as they are uniformly bounded  for $ |\lambda| \leq 1.$

Assuming $ |\lambda| \geq 1 $ we have, in view of \eqref{eq:pilambda},
$$\widehat{\mu_1}(\lambda) (I-\pi_\lambda(0,t)) \varphi(\xi) = (1-e^{i\lambda t}) \widehat{\mu_1}(\lambda)\varphi(\xi), \quad  \varphi \in L^2(\R^n).
$$
By mean value theorem, the operator norm of $ (1-e^{i\lambda t}) \widehat{\mu_1}(\lambda) $ is bounded by 
$$ C |t| |\lambda| \sup_{k} |\psi_k^{n-1}(\sqrt{|\lambda|})| \leq  C |t| |\lambda|^{-(n-1)+2/3}$$
where we have used the estimate in Lemma \ref{lem:uniform} (for $\alpha=1$). Thus for  $ n \geq 2 ,$ 
$$ \|\widehat{\mu_1}(\lambda) (I-\pi_\lambda(0,t)) \|_{L^2\to L^2} \leq C |t| \leq C |(z,t)|^2.
$$ 
In order to estimate $\widehat{\mu_1}(\lambda) (I-\pi_\lambda(|z|e_1,0))$ we recall that
$$  \pi_\lambda(|z|e_1,0)\varphi(\xi) = e^{i\lambda |z| \xi_1} \varphi(\xi), \quad   \varphi \in L^2(\R^n).
$$
Since we can write
$$ (1-e^{i\lambda |z| \xi_1}) = -i\lambda |z| \xi_1 \int_0^1 e^{it\lambda |z| \xi_1} dt = \lambda |z| \xi_1 m_\lambda(|z|, \xi) $$
with a bounded function $ m_\lambda(|z|,\xi), $ it is enough to estimate the norm of the operator $ |z|  \widehat{\mu_1}(\lambda) M_\lambda $ where $ M_\lambda \varphi (\xi) = \lambda \xi_1 \varphi(\xi).$

Let $ A(\lambda) = \frac{\partial}{\partial \xi_1}+|\lambda| \xi_1 $ and $ A(\lambda)^\ast = -\frac{\partial}{\partial \xi_1}+|\lambda| \xi_1 $ be the annihilation and creation operators, so that we can express $ M_\lambda $ as 
$ M_\lambda = \frac{1}{2} ( A(\lambda) + A(\lambda)^\ast) .$ Thus it is enough to consider $ |z| \widehat{\mu_1}(\lambda) A(\lambda) $ and $ |z| \widehat{\mu_1}(\lambda) A(\lambda)^\ast. $  Moreover, as the Riesz transforms 
$ H(\lambda)^{-1/2}A(\lambda) $ and $ H(\lambda)^{-1/2}A(\lambda)^\ast $ are bounded on $ L^2(\R^n) $ we only need to consider $ |z|\widehat{\mu_1}(\lambda) H(\lambda)^{1/2}.$ But the operator norm of $
\widehat{\mu_1}(\lambda) H(\lambda)^{1/2} $ is given by $\sup_{k} ((2k+n)|\lambda|)^{1/2} |\psi_k^{n-1}(\sqrt{|\lambda|})|$ which, in view of Lemma \ref{lem:uniform} (for $\alpha=1/2$), is bounded by $ C |\lambda|^{-(n-1)+2/3}.$ Thus for $ n \geq 2 $ we obtain
$$  \|\widehat{\mu_1}(\lambda) (I-\pi_\lambda(|z|e_1,0))\|_{L^2\to L^2}  \leq C |z| \leq C |(z,t)|.$$
This completes the proof of the proposition.
\end{proof}

\begin{rem}
\label{rem:restrn1}
Observe that the result above is restricted to the case $n\ge2$, and this is due to the restriction on the available sharp estimates for the Laguerre functions, see Lemmas~\ref{lem:T} and \ref{lem:uniform} (in particular,  we are using Lemma  \ref{lem:uniform} with $\delta=n-1$). We do not know whether there is a way to reach $n=1$ with our approach.
\end{rem}

\begin{cor}
\label{cor:A1}
Assume that $n\ge 2$. Then for $ y \in \H^n$, $|y| \leq 1$, and for $\big(\frac1p,\frac1q\big)$ in the interior of the triangle joining the points $(0,0), (1,1)$ and $\big(\frac{3n+1}{3n+4},\frac{3}{3n+4}\big)$, there exists $0<\nu<1$ such that we have the inequality
$$  
\| A_1 - A_1 \tau_y \|_{L^p\to L^q} \leq C |y|^{\nu},
$$
where $ \tau_y f(x) = f(xy^{-1}) $ is the right translation operator.

\end{cor}
\begin{proof}
The result follows by Riesz-Thorin interpolation theorem, taking into account Corollary \ref{cor:LpLq} and Proposition \ref{prop:continuity}. 
\end{proof}

We need a  version of the inequality in Corollary \ref{cor:A1} when $ A_1 $ is replaced by $A_r$. This can be easily achieved by making use of the following lemma which expresses $ A_r $ in terms of $ A_1$. Let $\delta_r\varphi(w,t)=\varphi(rw,r^2t)$ stand for the non-isotropic dilation on $ \H^n$.
\begin{lem}
\label{lem:dilat}
For any  $r> 0$  we have $A_rf=\delta_r^{-1}A_1\delta_rf.$
\end{lem}
\begin{proof}
This is just an easy verification.  Since
$$ A_1\delta_rf(z,t)  = \int_{|\omega| =1} f(rz-r\omega, r^2t-\frac{1}{2}r^2 \Im(z \cdot \bar{\omega})) d\mu_1(\omega)$$ it follows immediately
$$ (\delta_r^{-1} A_1 \delta_rf)(z,t) = \int_{|\omega| =1} f(z-r\omega, t-\frac{1}{2}r \Im(z \cdot \bar{\omega})) d\mu_1(\omega) = A_rf(z,t).$$  

\end{proof}

\begin{cor}
\label{cor:dilat2}
Assume that $n\ge 2$. Then for $ y \in \H^n, |y| \leq 1$, and for $\big(\frac1p,\frac1q\big)$ in the interior of the triangle joining the points $(0,0), (1,1)$ and $\big(\frac{3n+1}{3n+4},\frac{3}{3n+4}\big)$, there exists $0<\nu<1$ such that  we have the inequality
$$  
\| A_r - A_r \tau_y \|_{L^p\to L^q} \leq C r^{-\nu} |y|^{\nu}  r^{(2n+2)(\frac{1}{q}-\frac{1}{p})}.
$$
\end{cor}
\begin{proof} Observe that $ \delta_r( \tau_y f) = \tau_{\delta_r^{-1}y}(\delta_r f )$, which follows from the fact that $ \delta_r : \H^n \rightarrow \H^n $ is an automorphism. The corollary follows from Corollary \ref{cor:A1},  Lemma \ref{lem:dilat}, and the fact that $ \|\delta_r f\|_p = r^{-\frac{(2n+2)}{p}}$ for any $  1 \leq p < \infty.$
\end{proof}

\section{Sparse bounds for the lacunary spherical maximal function}
\label{sec:sparse}

Our aim in this section is to prove the sparse bounds for the lacunary spherical maximal function stated in Theorem \ref{thm:mainSH}. In doing so we closely follow \cite{Lacey} with suitable modifications that are necessary since we are dealing with a non-commutative set up. We can equip $ \H^n $ with a metric induced by the Koranyi norm which makes it a space of homogeneous type. On such spaces there is a well defined notion of dyadic cubes and grids with properties similar to their counter parts in the Euclidean setting. However, we need to be careful with the metric we choose since the group is non-commutative.

Recall that the Koranyi norm on $ \H^n $ is defined by $ |x| = |(z,t)| = (|z|^4+t^2)^{1/4}$, which is homogeneous of degree one with respect to the non-isotropic dilations. Since we are considering $ f \ast \mu_r $ it is necessary to work with the left invariant metric $ d_L(x,y) = |x^{-1}y| = d_L(0, x^{-1}y)$ instead of the standard metric $ d(x,y) = |xy^{-1}| = d(0, xy^{-1})$, which is right invariant. The balls and cubes are then defined using $ d_L $. Thus $ B(a,r) = \{ x \in \H^n: |a^{-1}x| < r \}.$ With this definition we note that $ B(a,r) =a\cdot B(0,r) $, a fact which is crucial. This allows us to conclude that when $ f $ is supported in $ B(a,r) $ then $ f \ast \mu_s $ is supported in $ B(a,r+s).$ Indeed, as the support of $ \mu_s $ is contained in $ B(0,s) $ we see that  $ f \ast \mu_s $ is supported in $ B(a,r)\cdot B(0,s) \subset a\cdot B(0,r)\cdot B(0,s) \subset B(a,r+s).$
\begin{thm}
\label{thm:homogrid}
Let ${\eta}\in (0,1)$ with ${\eta}\le 1/96$. Then there exists a countable set of points $\{z_{\nu}^{k,\alpha} : \nu\in \mathscr{A}_k\}$, $k\in \Z$, $\alpha=1,2,\ldots,N$ and a finite number of dyadic systems $\mathcal{D}^{\alpha}:=\cup_{k\in \Z}\mathcal{D}_k^{\alpha}$, $\mathcal{D}_k^{\alpha}:=\{Q_{\nu}^{k,\alpha}:\nu\in \mathscr{A}_k\}$ such that
\begin{enumerate}
\item For every $\alpha\in \{1,2,\ldots, N\}$ and $k\in \Z$ we have
\begin{itemize}
\item[i)] $\H^n=\cup_{Q\in \mathcal{D}_k^{\alpha}}Q$ (disjoint union).
\item[ii)] $Q,P\in \mathcal{D}^{\alpha}\Rightarrow Q\cap P\in \{\emptyset, P, Q\}$.
\item[iii)] $Q_{\nu}^{k,\alpha}\in \mathcal{D}^{\alpha}\Rightarrow B\big(z_{\nu}^{k,\alpha}, \frac{1}{12}{\eta}^k\big)\subseteq Q_{\nu}^{k,\alpha}\subseteq B\big(z_{\nu}^{k,\alpha}, 4{\eta}^k\big)$. In this situation $z_\nu^{k,\alpha} $ is called the center of the cube and the side length $\ell{ (Q_\nu^{k,\alpha})}$ is defined to be ${\eta}^k.$
\end{itemize}
\item For every ball $B=B(x,r)$, there exists a cube $Q_B\in \cup_{\alpha}\mathcal{D}^{\alpha}$ such that $B\subseteq Q_B$ and $\ell(Q_B)={\eta}^{k-1}$, where $k$ is the unique integer such that ${\eta}^{k+1}<r\le {\eta}^k.$ 
\end{enumerate}
\end{thm}
\begin{proof}
It follows from Theorem 4.1, the proof of Lemma 4.12, Remark 4.13 and Theorem~2.2 in \cite{HK}, where the choices $c_0=1/4$ and $C_0=2$ in \cite[Theorem 2.2]{HK} are made so that the property $(2)$ holds (see \cite[Lemma 4.10]{HK}). 
\end{proof}
We will first prove a lemma that is the analogue of \cite[Lemma 2.3]{Lacey}.
\begin{lem}
\label{lem:anal23}
Let $f_1$ and $f_2$ be supported on a cube $Q$ and let $\ell(Q)=r$. For $\big(\frac1p,\frac1q\big)$ in the interior of the triangle joining the points $(0,1), (1,0)$ and $ (\frac{3n+1}{3n+4},\frac{3n+1}{3n+4})$, there holds
$$
|\langle A_rf_1-A_r\tau_yf_1,f_2\rangle|\lesssim |y/r|^{\nu}|Q|\langle f_1\rangle_{Q,p}\langle f_2\rangle_{Q,q}, \quad |y|<r
$$
for some $\nu>0$.
\end{lem}
\begin{proof}
By H\"older's inequality and Corollary \ref{cor:dilat2} for the pair $\big(\frac1p,\frac1{q'}\big)$ we have, for $|y|<r$,
\begin{align*}
|\langle A_rf_1-A_r\tau_yf_1,f_2\rangle|&\le \|A_rf_1-A_r\tau_yf_1\|_{q'}\|f_2\|_{q}\\
&\le Cr^{(2n+2)(\frac{1}{q'}-\frac{1}{p})}r^{-\nu} |y|^{\nu}\|f_1\|_p\|f_2\|_{q}\\
&=Cr^{(2n+2)(\frac{1}{q'}-\frac{1}{p})}r^{-\nu} |y|^{\nu}|Q|^{\frac{1}{p}+\frac{1}{q}}\langle f_1\rangle_{Q,p}\langle f_2\rangle_{Q,q}\\
&\lesssim |Q|^{\frac{1}{q'}-\frac{1}{p}}|Q|^{\frac{1}{p}+\frac{1}{q}}r^{-\nu} |y|^{\nu}\langle f_1\rangle_{Q,p}\langle f_2\rangle_{Q,q}\\
&\lesssim |Q|r^{-\nu} |y|^{\nu}\langle f_1\rangle_{Q,p}\langle f_2\rangle_{Q,q},
\end{align*}
as $ |Q|$ is comparable to $ r^{2n+2}$.
\end{proof}

\begin{lem}
\label{lem:Hn}
Let $0<{\eta}< \frac{1}{96}$. For $Q$ with $\ell(Q)={\eta}^k$, $k\in \Z$, we consider
$$
\mathbb{V}_{Q}=\{P\in \mathcal{D}^1_{k+3}: B(z_{P},{\eta}^{k+1})\subseteq Q\}
$$
and define 
$$
A_{Q}f:=A_{{\eta}^{k+2}}(f{\bf{1}}_{V_Q})
$$ 
where $ V_{Q}=\cup_{P\in \mathbb{V}_{Q}}P.$ 
Then, for any $ f $ supported in $ Q $, the support  of $A_{Q}f $ is also contained in $ Q.$ Moreover,
$$A_{{\eta}^{k+2}}f\le \sum_{\alpha=1}^N\sum_{Q\in \mathcal{D}_k^{\alpha}}A_Q(f).$$
\end{lem}

\begin{proof}
Observe that for any $ x \in \H^n$  there exists $P\in \mathcal{D}^1_{k+3}$ such that $ x \in P \subseteq B(z_P, 4{\eta}^{k+3}).$  Then $ P \subseteq B(z_P,{\eta}^{k+1})\subseteq Q$ for some $Q$ in $\mathcal{D}_k^{\alpha}$, for some $\alpha$. Therefore $P\in V_{Q}$ and hence $ x \in V_{Q}$. This proves that  $ \H^n=\bigcup_{\alpha=1}^N\bigcup_{Q\in \mathcal{D}_k^{\alpha}}\mathbb{V}_Q $, hence we have $f\le \sum_{\alpha=1}^N\sum_{Q\in \mathcal{D}_k^{\alpha}}f{\bf{1}}_{V_Q}$, and consequently $A_{{\eta}^{k+2}}f\le \sum_{\alpha=1}^N\sum_{Q\in \mathcal{D}_k^{\alpha}}A_Qf$. It remains to be proved that $ A_Qf $ is supported in $ Q.$
Now assume that $\supp f\subseteq Q$ and recall  $A_{{\eta}^{k+2}}f(x)=f\ast \mu_{{\eta}^{k+2}}(x)$. Then it is enough to show that  $\supp A_{{\eta}^{k+2}}(f{\bf{1}}_P) \subseteq B(z_P, {\eta}^{k+1})$ for every $P \in \mathbb{V}_Q$. Indeed, 
$$\supp (f {\bf{1}}_P) \ast \mu_{{\eta}^{k+2}}\subseteq (\supp (f {\bf{1}}_P))\cdot(\supp \mu_{{\eta}^{k+2}})\subseteq z_P\cdot B(0,{\eta}^{k+2})\cdot B(0, \delta^{k+2})$$
which is contained in $  B(z_P, {\eta}^{k+1}) \subseteq Q $ by the definition of $ V_Q$. Observe that the above argument fails if we use balls defined by the standard right invariant metric. The lemma is proved.
\end{proof}

\begin{rem}
Actually we can take any ${\eta}>0$ when considering $A_{{\eta}^{k+2}}(f{\bf{1}}_{V_Q})$ in Lemma \ref{lem:Hn} (and in Theorems \ref{thm:spherical} and \ref{thm:mainSH}), in particular we could consider the means $A_{2^{-m}}(f{\bf{1}}_{V_Q})$, $m\in \Z$, or even more general $A_{\delta^{m}}(f{\bf{1}}_{V_Q})$, for any $\delta>0$. In that case we have to do some modifications in defining $A_Q f$, where one has to use the fact that if ${\eta}<\frac{1}{96}$ then the number of points of the form $\delta^m$, $m\in \Z$, lying between ${\eta}^j$ and ${\eta}^{j+1}$, $j\in \Z$, does not depend on $j$. 

Indeed, let $0<{\eta}< \frac{1}{96}$, that is fixed due to the fact that we are dealing with a space of homogeneous type. For $Q$ with $\ell(Q)={\eta}^k$ we consider
$$
\mathbb{V}_{Q}=\{P\in \mathcal{D}^1_{k+3}: B(z_{P},{\eta}^{k+1})\subseteq Q\}.
$$
Let, for any $\zeta>0$,
$$
\Theta_k^{\zeta}=\{m\in \Z: {\eta}^{k+3}<\zeta^m\le {\eta}^{k+2} \quad \text{ or } \quad    {\eta}^{k+2}<\zeta^m\le {\eta}^{k+3}\}
$$
and define 
$$
A_{Q}f:=\sum_{m\in \Theta_k^{\zeta}}A_{\zeta^{m}}(f{\bf{1}}_{V_Q})
$$ 
where $ V_{Q}=\cup_{P\in \mathbb{V}_{Q}}P.$ 
Suppose $ f $  is supported in $ Q $. Since the support of each $A_{\zeta^{m}}(f{\bf{1}}_{V_Q})$ is contained in $Q$, then the support  of $A_{Q}f $ is also contained in $ Q$. Moreover,
$$
\sup_{m\in \Theta_k^{\zeta}}A_{\zeta^{m}}f\le \sum_{\alpha=1}^N\sum_{Q\in \mathcal{D}_k^{\alpha}}A_Q(f).
$$
Now, as the number of terms in $\Theta_k^{\zeta}$ does not depend on $k$, $A_Qf$ satisfies Lemma \ref{lem:anal23} with constant independent of $k$. 

Observe also that in particular, we can choose $\zeta=1/2$ in the reasoning above, which is the standard lacunary case. 
In order to avoid additional notation, we just chose $\zeta={\eta}$ in Lemma \ref{lem:Hn}. Nevertheless, for the main results, we will keep the standard lacunary notation.
\end{rem}

In view of Lemma \ref{lem:Hn} it suffices to prove the sparse bound for each $M_{\mathcal{D}^{\alpha}}f=\sup_{Q\in \mathcal{D}^{\alpha}}A_Qf$ for $\alpha=1,2,\ldots,N$. To see this, recall that, by Theorem \ref{thm:homogrid}, for each fixed $ \alpha $ and $ k $ we have $\H^n=\cup_{Q\in \mathcal{D}_k^{\alpha}}Q$ (disjoint union). Since the  support of $ A_Qf $ is contained in $ Q $ it follows that 
$$ \sum_{Q\in \mathcal{D}_k^{\alpha}}A_Q(f) \leq M_{\mathcal{D}^{\alpha}}f .$$
Let us fix then $\mathcal{D}=\mathcal{D}^{\alpha}$. We will linearise the supremum. Let us assume that $f$ is supported in a cube $Q_0\in \mathcal{D}$, and let $\mathcal{D}(Q_0)$ be the collection of all dyadic subcubes of $Q_0$. We define 
$$
E_Q:=\big\{x\in Q : A_Qf(x)\ge \frac{1}{2} \sup_{P\in \mathcal{D}(Q_0)}A_Pf(x)\big\}
$$ 
for $Q\in\mathcal{D}(Q_0)$. Note that for any $ x \in \H^n $ there exists  a $Q \in \mathcal{Q}$ such that  
$$ 
A_Qf(x)\ge \frac{1}{2} \sup_{P\in \mathcal{D}(Q_0)} A_Pf(x) 
$$  
and hence $ x \in E_Q.$  If we define $B_Q=E_Q\setminus \cup_{Q'\supseteq Q}E_{Q'}$, then $\{B_Q: Q\in\mathcal{D}(Q_0)\}$ are disjoint and also, $\cup_{Q\in\mathcal{D}(Q_0)}B_Q=\cup_{Q\in\mathcal{D}(Q_0)}E_Q$. For $f_1,f_2>0$ it then follows that
\begin{align*}
\notag\langle\sup_{P\in\mathcal{D}(Q_0)}A_Pf_1,f_2\rangle&=\sum_{Q\in \mathcal{Q}} \int_{E_Q} \sup_{P\in\mathcal{D}(Q_0)}A_Pf_1(x)f_2(x)\,dx\\
\notag&\le 2 \sum_{Q\in\mathcal{D}(Q_0)}\int_{B_Q}A_Qf_1(x)f_2(x)\,dx\\
\notag&\le 2 \sum_{Q\in\mathcal{D}(Q_0)}\int_{\H^n}A_Qf_1(x)f_2(x)\mathbf{1}_{B_Q}(x)\,dx\\
&\le 2 \sum_{Q\in\mathcal{D}(Q_0)}\langle A_Qf_1,f_2\mathbf{1}_{B_Q}\rangle.
\end{align*}
Defining  $(f_2)_Q:=f_2\mathbf{1}_{B_Q}$ we will deal with $\sum_{Q\in\mathcal{D}(Q_0)}\langle A_Qf_1,(f_2)_Q\rangle$.

\begin{lem}
\label{lem:key}
Let $1<p,q<\infty$ be such that $\big(\frac1p,\frac1q\big)$ in the interior of the triangle joining the points $(0,1), (1,0)$ and $ (\frac{3n+1}{3n+4},\frac{3n+1}{3n+4})$. Let $f_1=\mathbf{1}_{F}$ and let $ f_2 $ be any bounded function supported in $ Q_0$. Let $C_0>1$ be a constant and let $\mathcal{Q}$ be a collection of dyadic subcubes of $Q_0\in \mathcal{D}$ for which the following holds
\begin{equation}
\label{eq:keyCondf1}
\sup_{Q'\in \mathcal{Q}}\sup_{Q: Q'\subset Q\subset Q_0}\frac{\langle f_1\rangle_{Q,p}}{\langle f_1\rangle_{Q_0,p}}<C_0.
\end{equation}
Then there holds 
$$
\sum_{Q\in \mathcal{Q}}\langle A_Qf_1,(f_2)_Q\rangle\lesssim \langle f_1\rangle_{Q_0,p}\langle f_2\rangle_{Q_0,q}|Q_0|.
$$
\end{lem}

\begin{proof}
We perform a Calder\'on--Zygmund decomposition of $f_1$ at height $2C_0\langle f_1\rangle_{Q_0,p}$. Let us denote by  $\mathcal{B}$  the resulting collection of (maximal) dyadic subcubes of $Q_0$  so that
\begin{equation}
\label{eq:stoppl}
\langle f_1\rangle_{Q,p}>2C_0\langle f_1 \rangle_{Q_0,p}.
\end{equation}
 Set $f_1=g_1+b_1$, where $\|g_1\|_{L^{\infty}}\lesssim \langle f_1\rangle_{Q_0,p}$ and
\begin{equation}
\label{eq:decoCZ}
b_1=\sum_{P\in \mathcal{B}}(f_1-\langle f_1\rangle_{P})\mathbf{1}_P=\sum_{j=s_0+1}^\infty \sum_{P\in \mathcal{B}(j)}(f_1-\langle f_1\rangle_{P})\mathbf{1}_P=:\sum_{j=s_0+1}^{\infty}B_{1,j},
\end{equation} 
where $\ell(Q_0)=\eta^{s_0}$ and $\mathcal{B}(j)=\{P\in \mathcal{B}: \ell(P)=\eta^{j}\}$. Now 
$$
\big|\sum_{Q\in \mathcal{Q}}\langle A_Qf_1,(f_2)_Q\rangle\big|\le \sum_{Q\in \mathcal{Q}}|\langle A_Qg_1,(f_2)_Q\rangle|+\sum_{Q\in \mathcal{Q}}|\langle A_Qb_1,(f_2)_Q\rangle|.
$$
Hence
$$
\sum_{Q\in \mathcal{Q}}|\langle A_Qg_1,(f_2)_Q\rangle|\lesssim \sum_{Q\in \mathcal{Q}}\|A_Qg_1\|_{\infty}\|f_2\mathbf{1}_{B_Q}\|_1\lesssim \langle f_1\rangle_{Q_0,p}\langle f_2\rangle_{Q_0,1}|Q_0|\lesssim \langle f_1\rangle_{Q_0,p}\langle f_2\rangle_{Q_0,q}|Q_0|.
$$
We now make the following useful observation. For all $Q\in \mathcal{Q}$ and $P\in \mathcal{B}$, if $P\cap Q\neq \emptyset$ then $P$ is properly contained in $ Q$. For otherwise, $Q\subseteq P$ and by the assumption on $ \mathcal{Q}$, we get $\langle f_1\rangle_{P,p}<C_0\langle f_1\rangle_{Q_0,p}$. But this contradicts the Calder\'on--Zygmund decomposition since  $ \langle f_1\rangle_{P,p}>2C_0\langle f_1 \rangle_{Q_0,p}$. Therefore, for any $Q\in \mathcal{Q}$ with $\ell(Q)=\eta^{s}$ we have
$$
\langle A_Qb_1,(f_2)_Q\rangle=\sum_{j>s}\langle A_QB_{1,j},(f_2)_Q\rangle=\sum_{j=1}^{\infty}\langle A_QB_{1,s+j},(f_2)_Q\rangle
$$
and so 
\begin{equation*}
\big|\sum_{Q\in \mathcal{Q}}\langle A_Qb_1,(f_2)_Q\rangle\big|\le \sum_{j=1}^{\infty}\sum_{Q\in \mathcal{Q}}|\langle A_QB_{1,s+j},(f_2)_Q\rangle|.
\end{equation*}
By making use of  the mean zero property of $b_1$, we see that
\begin{align*}
&|\langle A_QB_{1,s+j},(f_2)_Q\rangle|=|\langle B_{1,s+j},A_Q^*(f_2)_Q\rangle|\\
&= \sum_{P\in \mathcal{B}(s+j)}\big|\int_P A_Q^*(f_2)_Q(x)B_{1,s+j}(x)\,dx\big|\\
&\le  \sum_{P\in \mathcal{B}(s+j)}\frac{1}{|P|}\Big|\int_P\int_P\big[A_Q^*(f_2)_Q(x)-A_Q^*(f_2)_Q(x')\big]B_{1,s+j}(x)\,dx\,dx'\Big|.
\end{align*}
In the integral with respect to $ x'$ we make the change of variables $x'=xy^{-1}$ and note that $ P^{-1}x \subset P^{-1}P$ (here we have used the standard notation: for subsets $ A, B \subset \H^n$ we have $ AB = \{ a b : a \in A, b \in B \} $ and also $ A^{-1} = \{ a^{-1}: a \in A \}$). Since $ P \subset B(z_P, 4\eta^{s+j}) = z_P\cdot B(0,4\eta^{s+j})$ it follows that $ P^{-1} \subset  B(0, 4\eta^{s+j})z_P^{-1} $ and hence $ P^{-1}P  \subset P_0 = B(0, 8\eta^{s+j}) \subset B(0,\eta^{s+j-1})$ (observe that for the above argument it is important that the balls are defined using the left invariant metric). Thus we have
\begin{align*}
&|\langle A_QB_{1,s+j},(f_2)_Q\rangle|\\
&\le \sum_{P\in \mathcal{B}(s+j)}\frac{1}{|P|}\Big|\int_{P^{-1}P}\int_P\big[A_Q^*(f_2)_Q(x)-\tau_yA_Q^*(f_2)_Q(x)\big]B_{1,s+j}(x)\,dx\,dy\Big|\\
&\lesssim \frac{1}{|P_0|}\int_{P_0}\Big|\int_Q(f_2)_Q(x)(A_Q-A_Q\tau_{y^{-1}})B_{1,s+j}(x)\,dx\Big|\,dy\\
&\lesssim  \frac{1}{|P_0|}\int_{P_0}\Big|\frac{y}{\ell(Q)}\Big|^{\nu}|Q|\langle B_{1,s+j}\mathbf{1}_Q\rangle_{Q,p}\langle (f_2)_Q\rangle_{Q,q}\,dy\\
&\lesssim \frac{\eta^{(s+j-1)\nu}}{\eta^{s\nu}}|Q|\langle B_{1,s+j}\mathbf{1}_Q\rangle_{Q,p}\langle (f_2)_Q\rangle_{Q,q}\\
&=\eta^{(j-1)\nu}|Q|\langle B_{1,s+k}\mathbf{1}_Q\rangle_{Q,p}\langle (f_2)_Q\rangle_{Q,q},
\end{align*}
where we used Lemma \ref{lem:anal23} in the third inequality. 

Now we will prove 
\begin{equation}
\label{eq:claimII}
\sum_{Q\in \mathcal{Q}}|Q|\langle B_{1,s+j}\mathbf{1}_Q\rangle_{Q,p}\langle f_2\mathbf{1}_{B_Q}\rangle_{Q,q}\lesssim |Q_0|\langle f_1\rangle_{Q_0,p}\langle f_2\rangle_{Q_0,q},
\end{equation}
for all $j\ge1$ and for all $1<p,q<\infty$ such that $\big(\frac1p,\frac1q\big)$ are in the interior of the triangle joining the points $(0,1), (1,0)$ and $(1,1)$ (including the segment joining $(0,1)$ and $(1,0)$, excluding the endpoints). 

Let us fix the integer $j$. From the definition and \eqref{eq:keyCondf1} it follows that we can dominate
$$
|B_{1,s+j}|\lesssim \langle f_1\rangle_{Q_0,p}\mathbf{1}_{E_s}+\mathbf{1}_{F_{s}},
$$
where $E_s=E_{s,j}$ are pairwise disjoint sets in $Q_0$ as $s$ varies, and $F_{s}=F_{s,j}$ are pairwise disjoint sets in $F$. This produces two terms to control. For the first one, we will show that
\begin{equation}
\label{eq:first}
\langle f_1\rangle_{Q_0,p}\sum_{Q\in \mathcal{Q}}|Q|\langle \mathbf{1}_{E_s}\rangle_{Q,p}\langle f_2\mathbf{1}_{B_Q}\rangle_{Q,q}\lesssim |Q_0|\langle f_1\rangle_{Q_0,p}\langle f_2\rangle_{Q_0,q}.
\end{equation}

First we consider the case when $1/p+1/q=1,$  i.e. $ p =q'$, for $1<p<\infty$. 
\begin{align*}
&\sum_{Q\in \mathcal{Q}}|Q|\langle \mathbf{1}_{E_s}\rangle_{Q,p}\langle f_2\mathbf{1}_{B_Q}\rangle_{Q,p'}=\sum_{Q\in \mathcal{Q}}\Big(\int_Q\mathbf{1}_{E_s} \,dx\Big)^{1/p}\Big(\int_{Q}|f_2(x)|^{p'}\mathbf{1}_{B_Q}\,dx\Big)^{1/p'}\\
& \le \Big(\sum_{Q\in \mathcal{Q}}\int_Q\mathbf{1}_{E_s} \,dx\Big)^{1/p}\Big(\sum_{Q\in \mathcal{Q}}\int_{Q}|f_2(x)|^{p'}\mathbf{1}_{B_Q}\,dx\Big)^{1/p'}.
\end{align*}
On the one hand, from the disjointness of $B_Q$,
\begin{align*}
\sum_{Q\in \mathcal{Q}}\int_{Q}|f_2(x)|^{p'}\mathbf{1}_{B_Q}\,dx  = \int_{\cup B_Q}|f_2(x)|^{p'}\,dx&\leq \Big( \frac{1}{|Q_0|}\int_{Q_0}|f_2(x)|^{p'}\,dx \Big) |Q_0| \\
&= |Q_0|\langle f_2\rangle_{Q_0,p'}^{p'}.
\end{align*}
On the other hand,
as $E_s \cap Q $ are disjoint subsets of $ Q_0 $ we finally obtain 
$$
\sum_{Q\in \mathcal{Q}}\int_Q\mathbf{1}_{E_s} \,dx= \sum_{Q\in \mathcal{Q}}|E_s \cap Q|  \leq  |Q_0|.
$$
Thus the required inequality \eqref{eq:claimII} is proved for the first term in the case $ 1/p +1/q =1$. In the case $1/p +1/q =1+\tau>1$, set $1/\widetilde{p}=1/p-\tau$. Then, $1/\widetilde{p}+1/q=1$, and $p<\widetilde{p}$, so that
$$
\langle \mathbf{1}_{E_s}\rangle_{Q,p}\langle f_2\mathbf{1}_{B_Q}\rangle_{Q,q}\lesssim \langle \mathbf{1}_{E_s}\rangle_{Q,\widetilde{p}}\langle f_2\mathbf{1}_{B_Q}\rangle_{Q,q}.
$$
Then, \eqref{eq:first} follows from the previous case since $1/\widetilde{p}+1/q'=1$. 

Concerning the second term, we will show that
\begin{equation}
\label{eq:second}
\sum_{Q\in \mathcal{Q}}|Q|\langle \mathbf{1}_{F_{1,s}}\rangle_{Q,p}\langle f_2\mathbf{1}_{B_Q}\rangle_{Q,q}\lesssim |Q_0|\langle f_1\rangle_{Q_0,p}\langle f_2\rangle_{Q_0,q}.
\end{equation}
Again, the inequality holds in the case of $1/p+1/q=1$. For $1/p+1/q=1+\tau>1$, we define $\widetilde{p}$ as above. By using the stopping condition \eqref{eq:stoppl} we have then 
$$
\langle \mathbf{1}_{F_{1,s}}\rangle_{Q,p}\langle f_2\mathbf{1}_{B_Q}\rangle_{Q,q}\lesssim\langle \mathbf{1}_{F_{1}}\rangle_{Q_0}^{\tau}\langle \mathbf{1}_{F_{1,s}}\rangle_{Q,\widetilde{p}}\langle f_2\mathbf{1}_{B_Q}\rangle_{Q,q}. 
$$
From this and by using the previous case, since $1/\widetilde{p}+1/q=1$, we can conclude \eqref{eq:second}, and therefore \eqref{eq:claimII}. The proof is complete.
\end{proof}

Let us proceed to prove Theorem \ref{thm:mainSH}. We will state it also here, for the sake of the reading.

\begin{thm} 
\label{thm:main}
Assume $ n \geq 2$. Let $ 1 < p, q < \infty $ be such that $ (\frac{1}{p},\frac{1}{q}) $ belongs to the interior of the triangle joining the points $ (0,1), (1,0) $ and $ (\frac{3n+1}{3n+4},\frac{3n+1}{3n+4})$. Then for any  pair of compactly supported bounded functions $ (f_1,f_2) $ there exists a $ (p,q)$-sparse form such that $ \langle M_{\operatorname{lac}}f_1,f_2\rangle \leq C \Lambda_{\mathcal{S},p,q}(f_1,f_2)$.
\end{thm}

\begin{proof}
Fix a dyadic grid $\mathcal{D}$ and consider the maximal function 
$$
M_{\mathcal{D}}f_1(x)=\sup_{Q\in \mathcal{D}}|A_{Q}f_1(x)|.
$$
We can assume that $f_1$ is supported in $Q_0$ so that $A_Qf_1=0$ for all large enough cubes. According to this, we will therefore prove the sparse bound for the maximal function
$$
M_{\mathcal{D}\cap Q_0}f_1(x)=\sup_{Q\in \mathcal{D}}|A_{Q}f_1(x)|.
$$
From this, it follows that $M_{\operatorname{lac}}$ is bounded by the sum of a finite number of sparse forms. By the definition of supremum, given $f_1,f_2$, there is a sparse family of dyadic cubes $\mathcal{S}_0$ so that $\sup_{\mathcal{S}}\Lambda_{\mathcal{S},p,q}(f_1,f_2)\le 2 \Lambda_{\mathcal{S}_0,p,q}(f_1,f_2)$. Therefore, the claimed sparse bound holds.

As explained above, by linearising the supremum it is enough to prove the sparse bound for the sum
\begin{equation}
\label{eq:linear}
\sum_{Q\in\mathcal{D}\cap Q_0}\langle A_Qf_1,f_2\mathbf{1}_{B_Q}\rangle
\end{equation}
for the collection of pairwise disjoint $B_{Q}\subset Q$ described just before Lemma \ref{lem:key}.

Given $1<p,q<\infty$ so that Corollaries \ref{cor:LpLq} and \ref{cor:dilat2} hold for $\big(\frac1p,\frac{1}{q'}\big)$, we have to produce a sparse family $\mathcal{S}$ of subcubes of $Q_0$ such that
$$
\langle M_{\mathcal{D}\cap Q_0}f_1,f_2\rangle \le  2 \sum_{Q\in\mathcal{D}\cap Q_0}\langle A_Qf_1,f_2\mathbf{1}_{B_Q}\rangle \le C\sum_{S\in \mathcal{S}}|S|\langle f_1\rangle_{S,p}\langle f_2 \rangle_{S,q}
$$
where for each $S\in \mathcal{S}$, there exists $F_S\subset S$ with $|F_S|\ge \frac12 |S|$.

We first prove \eqref{eq:linear} when $f_1$ is the characteristic function of a set $F\subset Q_0$. Consider the collection $\mathcal{E}_{Q_0}$ of maximal children $P\subset Q_0$ for which
$$
\langle f_1\rangle_{P,p}>2\langle f_1\rangle_{Q_0,p}.
$$
 Let $E_{Q_0}=\cup_{P\in \mathcal{E}_{Q_0}}$. For a suitable choice of $c_n>1$ we can arrange $|E_{Q_0}|<\frac12|Q_0|$. We let $F_{Q_0}=Q_0\setminus E_{Q_0}$ so that $|F_{Q_0}|\ge \frac12|Q_0|$. We define
\begin{equation}
\label{eq:zero}
\mathcal{Q}_0=\{Q\in \mathcal{D}\cap Q_0: Q\cap E_{Q_0}=\emptyset\}.
\end{equation}
Note that when $Q\in \mathcal{Q}_0$ then $\langle f_1\rangle_{Q,p}\le 2\langle f_1\rangle_{Q_0,p}$. For otherwise, if $\langle f_1\rangle_{Q,p}>2\langle f_1\rangle_{Q_0,p}$ then there exists $P\in \mathcal{E}_{Q_0}$ such that $P\supset Q$, which is a contradiction. For the same reason, if $Q'\in \mathcal{Q}_0$ and $Q'\subset Q\subset Q_0$ then $\langle f_1\rangle_{Q,p}\le 2\langle f_1\rangle_{Q_0,p}$. Thus 
$$
\sup_{Q'\in \mathcal{Q}_0}\sup_{Q:Q'\subset Q\subset Q_0}\langle f_1\rangle_{Q,p}\le 2\langle f_1\rangle_{Q_0,p}.
$$
Note that for any $Q\in \mathcal{D}\cap Q_0$, either $Q\in \mathcal{Q}_0$ or $Q\subset P$ for some $P\in \mathcal{E}_{Q_0}$. Thus 
$$
\sum_{Q\in \mathcal{D}\cap Q_0}\langle A_Q f_1,f_2\mathbf{1}_{B_Q}\rangle=\sum_{Q\in Q_0}\langle A_Q  f_1,f_2\mathbf{1}_{B_Q}\rangle+\sum_{P\in \mathcal{E}_{Q_0}}\sum_{Q\subset P}\langle A_Q  f_1,f_2\mathbf{1}_{B_Q}\rangle
$$
for any $Q\in \mathcal{Q}_0$, $Q\subset F_{Q_0}$ and hence
$$
\sum_{Q\in \mathcal{Q}_0}\langle A_Q  f_1,f_2\mathbf{1}_{B_Q}\rangle=\sum_{Q\in \mathcal{Q}_0}\langle A_Q  f_1,f_2\mathbf{1}_{F_{Q_0}}\mathbf{1}_{B_Q}\rangle.
$$
Applying Lemma \ref{lem:key} we obtain
$$
\sum_{Q\in \mathcal{Q}_0}\langle A_Q  f_1,f_2\mathbf{1}_{B_Q}\rangle\le C|Q_0|\langle f_1\rangle_{Q_0,p} \langle f_2\mathbf{1}_{F_{Q_0}}\rangle_{Q_0,q}.
$$

Let $\{P_j\}$ be an enumeration of the cubes in $\mathcal{E}_{Q_0}$. Then the second sum above is given by 
$$
\sum_{j=1}^{\infty}\sum_{Q\in P_j\cap \mathcal{D}}\langle A_Q  f_1,f_2\mathbf{1}_{B_Q}\rangle.
$$
For each $j$ we can repeat the above argument recursively. Putting everything together we get a sparse collection $\mathcal{S}$ for which
\begin{equation}
\label{eq:toget}
\sum_{Q\in  \mathcal{D}\cap Q_0}\langle A_Q  f_1,f_2\mathbf{1}_{B_Q}\rangle\le C\sum_{S\in \mathcal{S}}|S| |\langle f_1\rangle_{S,p} \langle f_2\mathbf{1}_{F_{S}}\rangle_{S,q}.
\end{equation}
This proves the result when $f_1=\mathbf{1}_F$. We pause for a moment to remark that we have actually proved a sparse domination stronger than the one stated in the theorem. However, we are not able to prove such a result for general $ f_1$.

Now we prove the theorem for any bounded $f_1\ge0$ supported in $Q_0$. We start as in the case of $f_1=\mathbf{1}_F$ but now we define $\mathcal{Q}_0$ using stopping conditions on both $f_1$ and $f_2$. Thus we let $\mathcal{E}_{Q_0}$ stand for the collection of maximal subcubes $P$ of $Q_0$ for which either $\langle f_1\rangle_{P,p}>2\langle f_1\rangle_{Q_0,p}$ or $\langle f_2\rangle_{P,q}>2\langle f_2\rangle_{Q_0,q}$. As before, we define $E_{Q_0}=\cup_{P\in \mathcal{E}_{Q_0}}$ and $F_{Q_0}=Q_0\setminus E_{Q_0}$ so that $|F_{Q_0}|\ge \frac12|Q_0|$. We let 
$$
\mathcal{Q}_0=\{Q\in \mathcal{D}\cap Q_0: Q\cap E_{Q_0}=\emptyset\}.
$$
Then it follows that
$$
\sup_{Q'\in \mathcal{Q}_0}\sup_{Q:Q'\subset Q\subset Q_0}\langle f_1\rangle_{Q,p}\le 2\langle f_1\rangle_{Q_0,p}
$$
and
$$
\sup_{Q'\in \mathcal{Q}_0}\sup_{Q:Q'\subset Q\subset Q_0}\langle f_2\rangle_{Q,q}\le 2\langle f_2\rangle_{Q_0,q}.
$$
If we can show that 
\begin{equation}
\label{eq:star}
\sum_{Q\in \mathcal{Q}_0}\langle A_Q f_1,f_2\mathbf{1}_{B_Q}\rangle\le C|Q_0|\langle f_1\rangle_{Q_0,\rho} \langle f_2\rangle_{Q_0,q}
\end{equation}
for some $\rho>p$, then we can proceed as in the case of $f_1=\mathbf{1}_F$ to get the sparse domination
$$
\langle M_{\mathcal{D}}f_1\rangle\le  C\sum_{S\in \mathcal{S}}|S| |\langle f_1\rangle_{S,\rho} \langle f_2\rangle_{S,q}.
$$

In order to prove $\eqref{eq:star}$ we make use of the sparse domination already proved for $f_1=\mathbf{1}_F$. Defining $E_m:=\{x\in Q_0: 2^m\le f_1(x)\le 2^{m+1}\}$ and $f_{1,m}:=f_1\mathbf{1}_{E_m}$, we have the decomposition $f_1=\sum_mf_{1,m}$ (since $f_1$ is bounded it follows that $E_m=\emptyset$ for all $m\ge m_0$ for some $m_0\in \Z$). By applying the sparse domination \eqref{eq:toget} to $\mathbf{1}_{E_m}$ we obtain the following:
\begin{align*}
\sum_{Q\in \mathcal{Q}_0}\langle A_Qf_{1,m},f_2\mathbf{1}_{B_Q}\rangle&\le 2^{m+1}\sum_{Q\in \mathcal{Q}_0}\langle A_Q\mathbf{1}_{E_m},f_2\mathbf{1}_{B_Q}\rangle\\
&=2^{m+1}\sum_{Q\in \mathcal{Q}_0}\langle A_Q\mathbf{1}_{E_m},f_2\mathbf{1}_{F_{Q_0}}\mathbf{1}_{B_Q}\rangle\\
&\le2^{m+1}\sum_{Q\in Q_0\cap \mathcal{D}}\langle A_Q\mathbf{1}_{E_m},f_2\mathbf{1}_{F_{Q_0}}\mathbf{1}_{B_Q}\rangle\\
&\le C2^{m+1}\sum_{S\in \mathcal{S}_m}|S|\langle \mathbf{1}_{E_m}\rangle_{S,p}\langle f_2\mathbf{1}_{F_{Q_0}}\rangle_{S,q},
\end{align*}
where in the last three lines we used that for any $Q\in \mathcal{Q}_0$, $Q\subset F_{Q_0}$, \eqref{eq:zero} and \eqref{eq:toget}. In the above sum, $\langle f_2\mathbf{1}_{F_{Q_0}}\rangle_{S,q}=0$ unless $S\cap F_{Q_0}\neq \emptyset$. If $S\subset F_{Q_0}$ then by the definition of $\mathcal{Q}_0$ in \eqref{eq:zero} it follows that $S\in \mathcal{Q}_0$ and 
$$
\langle f_2\mathbf{1}_{F_{Q_0}}\rangle_{S,q}\le \langle f_2\rangle_{S,q}\le c_n \langle f_2\rangle_{Q_0,q}.
$$
If $S\cap F_{Q_0}\neq \emptyset$ as well as $S\cap E_{Q_0}\neq \emptyset$, then for some $P\in \mathcal{E}_{{Q}_0}$, $P\subset S$. But then by the maximality of $P$ we have
$$
\langle f_2\mathbf{1}_{F_{Q_0}}\rangle_{S,q}\le \langle f_2\rangle_{S,q}\le 2 \langle f_2\rangle_{Q_0,q}.
$$
Using this we obtain
$$
\sum_{Q\in \mathcal{Q}_0}\langle A_Q f_{1,m},f_2\mathbf{1}_{B_Q}\rangle\le C2^{m+1}\langle f_2\rangle_{Q_0,q}\sum_{S\in \mathcal{S}_m}|S|\langle \mathbf{1}_{E_m}\rangle_{S,p}.
$$
By Lemma \ref{lem:carleson} we get
$$
\sum_{Q\in \mathcal{Q}_0}\langle A_Q f_{1,m},f_2\mathbf{1}_{B_Q}\rangle\le C2^{m+1}\langle f_2\rangle_{Q_0,q}\langle \mathbf{1}_{E_m}\rangle_{Q_0,\rho_1}|Q_0|
$$
for some $\rho_1>p$. As $f_1=\sum_m f_{1,m}$ it follows that 
$$
\sum_{Q\in \mathcal{Q}_0}\langle A_Q f_{1,m},f_2\mathbf{1}_{B_Q}\rangle\le C\langle f_2\rangle_{Q_0,q}|Q_0|\sum_m2^m\langle \mathbf{1}_{E_m}\rangle_{Q_0,\rho_1}.
$$
We now claim that (see Lemma \ref{lem:lorentz} below)
\begin{equation}
\label{eq:twostar}
\sum_m2^m\langle \mathbf{1}_{E_m}\rangle_{Q_0,\rho_1}\le C\|f_1\|_{L^{\rho_1,1}(Q_0,d\mu)}
\end{equation}
where $L^{\rho_1,1}(Q_0,d\mu)$ stands for the Lorentz space defined on the measure space $(Q_0,d\mu)$, $d\mu=\frac{1}{|Q_0|}dx.$  We also know that on a probability space, the $L^{\rho_1,1}(Q_0,d\mu)$ norm is dominated by the $L^{\rho}(Q_0,d\mu)$ norm for any $\rho>\rho_1$ (Lemma \ref{lem:proba}). Using these two results we see that
$$
\sum_{Q\in \mathcal{Q}_0}\langle A_Q f_1,f_2\mathbf{1}_{B_Q}\rangle\le C\langle f_2\rangle_{Q_0,q}|Q_0|\langle f_1\rangle_{Q_0,\rho}.
$$
Hence \eqref{eq:star} is proved and thus completes the proof of Theorem \ref{thm:main}.
\end{proof}

It remains to prove Lemma \ref{lem:proba} and the claim \eqref{eq:twostar}. The first one is a well known fact which we include here for the sake of completeness.
\begin{lem}
\label{lem:proba}
On a probability space $ (X,d\mu)$, $L^p(X,d\mu)\subset L^{r,1}(X,d\mu)  \text{ for } p>r.$
\end{lem}
\begin{proof}
Recall that the Lorentz spaces $ L^{p,q}(X,d\mu) $ are defined in terms of the Lorentz norms (see \cite{GM1})
$$ \| f\|_{p,q} =  
\begin{cases}
\Big(\int_0^{\infty}\big(t^{\frac{1}{p}}f^*(t)\big)^q\frac{dt}{t}\Big)^{\frac{1}{q}} \quad &\text{ if } q<\infty,\\
\sup_{t>0}t^{\frac{1}{p}}f^*(t)\quad &\text{ if } q=\infty,
\end{cases}
$$
where $ f^*(t) $ stands for the non-decreasing rearrangement of $ f.$  When  $f\in L^p(X,d\mu)$, as $d\mu$ is a probability measure, we know that the distribution function $df(s)$ of $f$ is bounded by $1$ and hence  $f^*(t)=0$ for $t\ge 1$. Now 
$$
\|f\|_{L^{r,1}(X,d\mu)}=\int_0^{\infty}t^{\frac{1}{r}-1}f^*(t)\,dt=\int_0^{1}t^{-\frac{1}{r'}}f^*(t)\,dt.
$$
By H\"older's inequality
$$
\|f\|_{L^{r,1}(X,d\mu)}\le \Big(\int_0^1 t^{-\frac{p'}{r'}} dt\Big)^{1/p'} \Big(\int_0^1f^*(t)^p\,dt\Big)^{1/p} =C_{r,p}\Big(\int_0^1f^*(t)^p\,dt\Big)^{1/p}
$$
where $C_{r,p}<\infty$ since $p'<r'$. This proves the claim since 
$$
\Big(\int_0^1f^*(t)^p\,dt\Big)^{\frac1p}=\|f\|_{L^p(X,d\mu)}.
$$
\end{proof} 

The claim \eqref{eq:twostar} is the content of the next lemma. 
\begin{lem}
\label{lem:lorentz}
Let  $f=\sum_mf_m$, $f_m=f\mathbf{1}_{E_m}$ where $E_m=\{x\in Q_0: 2^m\le |f(x)|\le 2^{m+1}\}.$ We  consider the probability measure $ d\mu = |Q_0|^{-1} dx $ on $ X = Q_0.$ Then for any $ r >1 $ we have
$$ \sum_m2^m\langle \mathbf{1}_{E_m}\rangle_{Q_0,r}\le C\|f\|_{L^{r,1}(Q_0,d\mu)}. $$  
\end{lem}
\begin{proof} 

We recall the following definition of the Lorentz norm in terms of $df(s)$:
$$
\|f\|_{L^{r,1}(X,d\mu)}=\int_0^{\infty}df(s)^{\frac{1}{r}}\,ds.
$$
As $df(s)$ is a decreasing function of $s$ we have
\begin{align*}
\|f\|_{L^{r,1}(X,d\mu)}&= \sum_m\int_{2^m}^{2^{m+1}}df(s)^{\frac{1}{r}}\,ds\\
&\ge \sum_mdf(2^{m})^{\frac1r}(2^{m+1}-2^m)\\
&=\frac12\sum_mdf(2^m)^{\frac1r}2^m.
\end{align*}
On the other hand, as $f_m=f\mathbf{1}_{E_m}$, it follows that  $
\mu(E_m) =df(2^m)-df(2^{m+1})\le df(2^m)$
and consequently,
$$
\sum_m \mu(E_m)^{\frac1r}2^m\le \sum_mdf(2^m)^{\frac1r}2^m\le 2\|f\|_{L^{r,1}(X,d\mu)}.
$$
This proves the lemma.
\end{proof}

In proving  Theorem \ref{thm:main} we have made use of the following lemma, which is proved  in \cite[Proposition 2.19]{Lacey}.  We include a proof  here for the convenience of the reader.
\begin{lem}[\cite{Lacey}]
\label{lem:carleson}
Let $\mathcal{S}$ be a collection of sparse subcubes  of a fixed dyadic cube $Q_0$ and let $1\le s<t<\infty$. Then, for a bounded function $\phi$,
$$
\sum_{Q\in \mathcal{S}}\langle \phi\rangle_{Q,s}|Q|\lesssim \langle \phi\rangle_{Q_0,t}|Q_0|.
$$
\end{lem}
\begin{proof}
By sparsity,
\begin{align*}
\sum_{Q\in \mathcal{S}}\langle \phi\rangle_{Q,s}|Q|&=\sum_{Q\in \mathcal{S}}\langle \phi\rangle_{Q,s}|Q|^{1/t+1/t'}\\
&\le \Big(\sum_{Q\in \mathcal{S}}\langle \phi\rangle_{Q,s}^t|Q|\Big)^{1/t} \Big(\sum_{Q\in \mathcal{S}}|Q|\Big)^{1/t'}\\
&\lesssim \Big(\sum_{Q\in \mathcal{S}}\langle |\phi|^s\rangle_{Q}^{t/s}|Q|\Big)^{1/t}|Q_0|^{1/t'}\\
&\lesssim \|\phi\mathbf{1}_{Q_0}\|_t|Q_0|^{1/t'}.
\end{align*}
\end{proof}

\section{Boundedness properties  for the lacunary spherical maximal function}
\label{sec:bound}

Consequences inferred from sparse domination are well-known and have been studied in the literature. We refer to \cite[Section 4]{BC} for an account of the same. In particular, sparse domination provides unweighted and weighted inequalities for the operators under consideration.

The strong boundedness is a result by now standard, see \cite{CUMP}, also \cite[Proposition 6.1]{Lacey}. Our Theorem \ref{thm:spherical} follows from Theorem \ref{thm:mainSH} and Proposition \ref{prop:un}.
\begin{prop}[\cite{CUMP}]
\label{prop:un}
Let $1\le r<s'\le\infty$. Then,
$$
\Lambda_{r,s}(f_1,f_2)\lesssim \|f_1\|_{L^p}\|f_2\|_{L^{p'}}, \quad r<p<s'.
$$
\end{prop}

For the sake of completeness we reproduce the proof, which is quite simple:  as the collection $ \mathcal{S} $ is  sparse, we have
$$ \Lambda_{r,s}(f_1,f_2) \leq C \sum_{S \in \mathcal{S}} \int_{E_S} \langle f_1\rangle_{S,r} \langle f_2 \rangle_{S,s} \mathbf{1}_{E_S} dx$$
where $ E_S \subset S $ are disjoint with the property that $ |E_S| \geq \eta |S|.$ The above leads to the estimate
$$ \Lambda_{r,s}(f_1,f_2) \leq C \int_{\H^n}  \big(M_{\operatorname{HL}} |f_1|^r(x)\big)^{1/r} \big(M_{\operatorname{HL}} |f_2|^s(x)\big)^{1/s} dx $$
where $M_{\operatorname{HL}}h $ stands for  the Hardy-Littlewood maximal function of $ h.$ In view of  the boundedness of $M_{\operatorname{HL}}$, an application of H\"older's inequality completes the proof of the proposition. 

A weight $w$ is a non-negative locally integrable function defined on $\H^n$.  Given $1<p<\infty$, the Muckhenhoupt class of weights $A_p$ consists of all $w$ satisfying
$$
[w]_{A_p}:=\sup_Q \langle w\rangle_{Q}\langle \sigma \rangle_{Q}^{p-1}<\infty,\quad \sigma:= w^{1-p'}
$$
where the supremum is taken over all cubes $Q$ in $\H^n$.  On the other hand, a weight $w$ is in the reverse H\"older class $\operatorname{RH}_p$, $1\le p<\infty$, if
$$
[w]_{\operatorname{RH}_p}=\sup_Q\langle w\rangle_Q^{-1}\langle w\rangle_{Q,p}<\infty,
$$
again the supremum taken over all cubes in $\H^n$.

The following theorem was shown in \cite[Section 6]{BFP}.

\begin{thm}[\cite{BFP}]
\label{thm:BFP}
Let $1\le p_0<q_0'\le\infty$. Then,
$$
\Lambda_{p_0,q_0}(f_1,f_2)\le \{[w]_{A_{p/p_0}}\cdot[w]_{\operatorname{RH}_{(q_0'/p)'}}\}^{\alpha}\|f_1\|_{L^p(w)}\|f_2\|_{L^{p'}(\sigma)}, \quad p_0<p<q_0',
$$
with $\alpha=\max\Big\{\frac{1}{p-1},\frac{q_0'-1}{q_0'-p}\Big\}$. 
\end{thm}

In view of Theorem \ref{thm:BFP} and Theorem \ref{thm:mainSH} we can obtain the following corollary: it provides unprecedented weighted estimates for the lacunary maximal spherical means in $\H^n$. We only state a qualitative result in order to simplify the presentation. 
\begin{cor}
\label{cor:weight}
Let $n\ge2$ and define
$$
\frac{1}{\phi(1/p_0)}=\begin{cases}1-\frac{1}{p_0}\frac{3}{3n+1}, \quad 0<\frac1p_0\le \frac{3n+1}{3n+4},\\
\frac{3n+1}{3}\Big(1-\frac1p_0\Big), \quad  \frac{3n+1}{3n+4}<\frac1p_0<1.
\end{cases}
$$
Then $M_{\operatorname{lac}}$ is bounded on $L^p(w)$ for $w\in A_{p/p_0}\cap \operatorname{RH}_{(\phi(1/p_0)'/p)'}$ and all $1<p_0<p<(\phi(1/p_0))'$.
\end{cor}

\section{The full maximal function}
\label{sec:full}

As in the case of  the lacunary spherical maximal function we can also deduce sparse bounds for the full maximal function. 

\begin{thm} 
\label{thm:sparseF}
Assume $ n \geq 2 $. Let $ 1 < p, q < \infty $ be such that $ (\frac{1}{p},\frac{1}{q}) $ belongs to the interior of the triangle joining the points $ (0,1)$, $ (\frac{2n-1}{2n},\frac{1}{2n})$ and $ (\frac{3n+1}{3n+7},\frac{3n+1}{3n+7})$. Then for any  pair of compactly supported bounded functions $ (f_1,f_2) $ there exists a $ (p,q)$-sparse form such that 
$$ \langle M_{\operatorname{full}}f_1,f_2\rangle \leq C \Lambda_{\mathcal{S},p,q}(f_1,f_2).
$$
\end{thm} 
Weighted norm inequalities for the full maximal function are implied from the sparse domination result, see Subsection \ref{sec:sparseF}. As explained in the Introduction, we expect that the range will not be sharp.

We will make use of the fixed time estimates for the operator $A_r$ from Section \ref{sec:Lp} trivially integrating in the $r$-variable and the known $ L^p$ estimates for $
M_{\operatorname{full}}$ to show the following theorem for the local (full) maximal operator. This theorem will be later used to prove Theorem~\ref{thm:sparseF}. Let us define, for some $0<\delta<1$,
$$ 
M_{\delta} f(z,t) := \sup_{1 \leq r \leq \delta^{-1}} |A_rf(z,t)|.
$$ 

\begin{thm}
\label{thm:LPQ}
Assume that $n\ge 2$ and  $0<\delta<1$. Then 
$$
M_{\delta} :L^{p}(\H^n)\to L^{q}(\H^n)
$$
whenever $\big(\frac{1}{p},\frac{1}{q}\big)$ lies in the interior of the triangle joining the points $(0,0)$, $\big(\frac{2n-1}{2n},\frac{2n-1}{2n} \big)$ and $\big(\frac{3n+1}{3n+7},\frac{6}{3n+7}\big)$, as well as the straight line segment joining the points $(0,0),\big(\frac{2n-1}{2n},\frac{2n-1}{2n} \big)$.
\end{thm}

Let $0<\delta <1$.
Observe that,  for $ f\ge 0$, by Fundamental Theorem of Calculus we obtain
$$
\sup_{1\le r\le \delta^{-1}}|A_rf|^q=|A_1f|^q+\int_1^{\delta^{-1}}q|A_rf|^{q-1}\Big(\frac d{dr}A_rf\Big)\,dr. 
$$
H\"older's inequality twice gives, for any $q\ge1$, 
\begin{equation}
\label{eq:FTC}
 \|M_{\delta}f\|_{L^q(\H^n)}  \leq  \|A_1f\|_{L^q(\H^n)}+\|A_rf\|_{L^q(\H^n\times [1,\delta^{-1}])}^{1-1/q}\Big\|\frac d{dr}A_rf\Big\|_{L^q(\H^n\times [1,\delta^{-1}])}^{1/q}.
\end{equation}
Since we already have $L^p -L^q $ estimates for $A_1f$ (see Corollary \ref{cor:LpLq}), we will only deal with $B_rf:=\frac d{dr}A_rf $ to get the $ L^p-L^q$ estimates for $M_{\delta}f$ in Theorem~\ref{thm:LPQ}. 

We will also use Corollary \ref{cor:dilat2} in order to prove a continuity condition of $M_{\delta}f$. The program is completely analogous to the lacunary case, only requiring more technical effort. We will omit the details in many instances. 

\subsection{$L^p-L^q$ estimates for the derivatives of the spherical means}
\label{sec:Lpfull}

Recall the expansion of $A_rf$ in \eqref{eq:expression}.
Now using the fact that $B_rf = \frac d{dr}A_rf $, we get the following expression for $B_rf $ 
\begin{multline}
\label{eq: expressionB}
B_rf(z,t)=(2\pi)^{-n-1}\\
\times\int_{-\infty}^{\infty}e^{-i\lambda t}\Big(\sum_{k=0}^{\infty}\big(-\frac{|\lambda| r}{2}\psi_k^{n-1}(\sqrt{|\lambda|}r)-\frac{k|\lambda|r}{n}\psi_{k-1}^{n}(\sqrt{|\lambda|}r)\big)f^{\lambda}\ast_{\lambda}\varphi_k^{\lambda}(z)\Big)|\lambda|^n\,d\lambda.
\end{multline}
where we have used the fact that $\frac{d}{dr}L_{k}^{\alpha}(r)=-L_{k-1}^{\alpha+1}$, see \cite[Chapter V]{Sz}. For $u>0$, we define
\begin{multline}
\label{eq:Abeta2}
\mathcal{B}_{u}^{\beta}f(z,t)=(2\pi)^{-n-1}\int_{-\infty}^{\infty}e^{-i\lambda t}\Big(\sum_{k=0}^{\infty}\big(-\frac{|\lambda|u}{2}\psi_k^{2\beta+n-1}(\sqrt{|\lambda|}u)\\
-\frac{k|\lambda|u}{n}\psi_k^{2\beta+n}(\sqrt{|\lambda|}u)\big)f^{\lambda}\ast_{\lambda}\varphi_k^{\lambda}(z)\Big)|\lambda|^n\,d\lambda,
\end{multline}
for $\Re(2\beta+n-1)>-1$. If $\beta=0$ we have $\mathcal{B}_{u}^{0}=B_u$. 

Let us consider also a rescaling of the formula \eqref{eq:Abetaa} namely, the operator $(\mathcal{A}_{\beta})_u$ given by 
$$
(\mathcal{A}_{\beta})_uf(z,t)=2\frac{\Gamma(\beta+n)}{\Gamma(\beta)\Gamma(n)}\int_0^1s^{2n-1}(1-s^2)^{\beta-1}P_{u^2(1-s^2)}f\ast \mu_{us}(z,t)\,ds. 
$$
Next we use the above expression along with the fact that $ \mathcal{B}_u^{\beta}f(z,t) = \frac d{du} (\mathcal{A}_{2\beta})_uf(z,t) $ to get the following expression for $\mathcal{B}_{u}^{\beta}f $.
\begin{lem}
\begin{align*}
\label{eq:beta }
& \mathcal{B}_u^{\beta}f (z,t)= 2\frac{\Gamma(2\beta+n)}{\Gamma(2\beta)\Gamma(n)u}\\
&\times\int_0^1s^{2n-1}(1-s^2)^{2\beta-1}\Big(c_n+\frac{2s(2\beta-1)}{1-s^2}\Big)P_{u^2(1-s^2)}f\ast \mu_{us}(z,t)\,ds\\
&\qquad -2\frac{\Gamma(2\beta+n)}{\Gamma(2\beta)\Gamma(n)}\int_0^1s^{2n-1}(1-s^2)^{2\beta-1}Q_{u^2(1-s^2)}f\ast \mu_{us}(z,t)\,ds\\
& \qquad +  2\frac{\Gamma(2\beta+n)}{\Gamma(2\beta)\Gamma(n)u}\int_0^1s^{2n}(1-s^2)^{2\beta-1}\frac{d}{ds}P_{u^2(1-s^2)}f\ast \mu_{us}(z,t)\,ds, 
\end{align*}
where $Q_{s}f:=f\ast q_{s}$ and $q_{s}(t)=c\frac{s^3}{(s^2+16t^2)^2}$ for some positive constant $c$.
\end{lem}
\begin{proof}
From Lemma \ref{lem:family} we get
\begin{align*}
 \mathcal{B}_u^{\beta}f (z,t)&= 2\frac{\Gamma(2\beta+n)}{\Gamma(2\beta)\Gamma(n)}\int_0^1s^{2n-1}(1-s^2)^{2\beta-1}\frac d{du}P_{u^2(1-s^2)}f\ast \mu_{us}(z,t)\,ds\\
 & +  2\frac{\Gamma(2\beta+n)}{\Gamma(2\beta)\Gamma(n)}\int_0^1s^{2n-1}(1-s^2)^{2\beta-1}P_{u^2(1-s^2)}f\ast \frac d{du}\mu_{us}(z,t)\,ds. 
\end{align*}
Define $$\mathcal{B}_u^{1,\beta}f (z,t) = 2\frac{\Gamma(2\beta+n)}{\Gamma(2\beta)\Gamma(n)}\int_0^1s^{2n-1}(1-s^2)^{2\beta-1}\frac d{du}P_{u^2(1-s^2)}f\ast \mu_{us}(z,t)\,ds. $$ and 
$$\mathcal{B}_u^{2,\beta}f (z,t) = 2\frac{\Gamma(2\beta+n)}{\Gamma(2\beta)\Gamma(n)}\int_0^1s^{2n-1}(1-s^2)^{2\beta-1}P_{u^2(1-s^2)}f\ast \frac d{du}\mu_{us}(z,t)\,ds.$$
Now $ \frac d{du} P_{u^2(1-s^2)} = f \ast \frac d{du} p_{u^2(1-s^2)}$. Hence
$$\frac d{du} P_{u^2(1-s^2)} = f \ast \frac 1u( p_{u^2(1-s^2)} - q_{u^2(1-s^2)}).
$$
Thus
\begin{multline*} 
\mathcal{B}_u^{1,\beta}f (z,t)\\
= 2\frac{\Gamma(2\beta+n)}{\Gamma(2\beta)\Gamma(n)}\int_0^1s^{2n-1}(1-s^2)^{2\beta-1}\frac 1u (P_{u^2(1-s^2)}-Q_{u^2(1-s^2)})f\ast \mu_{us}(z,t)\,ds. 
\end{multline*}
Also we have
\begin{align*}
&\mathcal{B}_u^{2,\beta}f (z,t) = 2\frac{\Gamma(2\beta+n)}{\Gamma(2\beta)\Gamma(n)}\int_0^1s^{2n-1}(1-s^2)^{2\beta-1}P_{u^2(1-s^2)}f\ast \frac{s}{u} \frac d{ds}\mu_{us}(z,t)\,ds\\
& =- 2\frac{\Gamma(2\beta+n)}{u \Gamma(2\beta)\Gamma(n)}\int_0^1 2ns^{2n-1}(1-s^2)^{2\beta-1}P_{u^2(1-s^2)}f\ast \mu_{us}(z,t)\,ds\\ 
&\qquad + 2\frac{\Gamma(2\beta+n)}{u \Gamma(2\beta)\Gamma(n)}\int_0^1 s^{2n}(2\beta-1)(2s)(1-s^2)^{2\beta-2}P_{u^2(1-s^2)}f\ast \mu_{us}(z,t)\,ds\\
&\qquad - 2\frac{\Gamma(2\beta+n)}{u \Gamma(2\beta)\Gamma(n)}\int_0^1 r^{2n}(1-s^2)^{2\beta-1}\frac d{ds}P_{u^2(1-s^2)}f\ast \mu_{us}(z,t)\,ds.
\end{align*}
Adding $\mathcal{B}_{u}^{1,\beta}f $ and  $\mathcal{B}_{u}^{2,\beta}f $, we get the required result.
\end{proof}

We define a new family by
$$
\mathcal{T}_{\beta}f=\mathcal{B}_1^{\beta}(f\ast_3 k_{2\beta})
$$ 
where $k_{\beta}$ is as in \eqref{eq:kbeta}.
For $\beta=0$, we have $\mathcal{T}_{\beta}=\mathcal{B}_1^0=B_1$. Analogously as in Lemma \ref{lem:explicitTb} we can prove the following. 

\begin{lem}
The operator $\mathcal{T}_{\beta}f$ has the explicit expansion
\begin{multline*}
\mathcal{T}_{\beta}f(z,t)=(2\pi)^{-n-1}\int_{-\infty}^{\infty}e^{-i\lambda t}(1-i\lambda)^{-2\beta}\Big(\sum_{k=0}^{\infty}\Big(-\frac{|\lambda| }{2} \psi_k^{2\beta + n-1}(\sqrt{|\lambda|})\\
 - \frac{k}{n}(|\lambda|)\psi_{k-1}^{2\beta +n}(\sqrt{|\lambda|}) \big)f^{\lambda}\ast_{\lambda}\varphi_k^{\lambda}(z)\Big)|\lambda|^n\,d\lambda.
\end{multline*}
\end{lem}

The next step is to show  that when $\beta=1+i\gamma$, $\mathcal{T}_{\beta}$ is bounded from $L^p(\H^n)$ into $L^{\infty}(\H^n)$ for any $p>1$, and that for certain negative values of $\beta$, $\mathcal{T}_{\beta}$ is bounded on $L^2(\H^n)$. 
\begin{prop}
\label{prop:LinftyF}
Let $n\ge 1$. For any $\delta>0$, $\gamma\in \R$,
$$
\|\mathcal{T}_{1+i\gamma}f\|_{\infty}\le C_1(\gamma)\|f\|_{1+\delta},
$$
where $C_1(\gamma)$ is of admissible growth.
\end{prop}
\begin{proof}
Let $\psi(t)=te^{-t}\chi_{(0,\infty)}(t)$. For $\beta=1+i\gamma$ it follows that 
\begin{align*}
&|\mathcal{T}_{1+i\gamma}f|  =|\mathcal{B}_1^{1+i\gamma}(f\ast_3 k_{2(1+i\gamma)} )(z,t)|\\
 &\le 2c_{n,\gamma}\frac{|\Gamma(2+2i\gamma+n)|}{\lvert\Gamma(2+2i\gamma)\rvert^2\Gamma(n)}\int_0^1s^{2n-1}P_{(1-s^2)}(f\ast_3 \psi )\ast \mu_{s}(z,t)\,ds\\
&\qquad + 2 \frac{|\Gamma(2+2i\gamma+n)|}{|\Gamma(2+2i\gamma)|^2\Gamma(n)}\int_0^1s^{2n-1}Q_{(1-s^2)}(f\ast_3 \psi)\ast \mu_{s}(z,t)\,ds\\
& \qquad  +  2\frac{|\Gamma(2+2i\gamma+n)|}{|\Gamma(2+2i\gamma)|^2\Gamma(n)}\int_0^1s^{2n}(1-s^2)| \frac{d}{ds}P_{(1-s^2)}(f\ast_3 \psi)|\ast \mu_{s}(z,t)\,ds.
\end{align*}
Let us  denote the right hand side of the above equation as $ I_1 +I_2+I_3$. 
 Since $\psi\ge0$, we get
$$
P_{1-r^2}(f\ast_3 \psi)=\psi\ast_3p_{1-r^2}\ast_3f\le \psi\ast_3M_{\operatorname{HL}}^0 f
$$ 
and 
$$
Q_{1-r^2}(f\ast_3 \psi)=\psi\ast_3q_{1-r^2}\ast_3f\le \psi\ast_3M_{\operatorname{HL}}^0f.
$$
  Thus we have the estimate
$$
|I_1+I_2 |\le C_1(\gamma)\int_0^1 (M_{\operatorname{HL}}^0 f\ast_3 \psi)\ast \mu_{s}(z,t)s^{2n-1}\,ds.
$$
Analogously as in the proof of Proposition \ref{prop:Linfty} we deduce
$$
|I_1+I_2|\le C_1(\gamma)M_{\operatorname{HL}}^0 f\ast K(z,t)
$$
where $K(z,t)=\chi_{|z|\le 1}(z)\psi(t)$. As $M_{\operatorname{HL}}^0f\in L^{1+\delta}(\H^n)$ and $K\in L^q(\H^n)$ for any $q\ge 1$, by H\"older we get
$$
\|I_1+I_2\|_{\infty}\le C_1(\gamma)\|M_{\operatorname{HL}}^0 f\|_{1+\delta}\le C_1(\gamma)\|f\|_{1+\delta}.
$$
Also we have 
$$
| I_3| \le 2\frac{| \Gamma(2+2i\gamma)+n)|}{|\Gamma(2+ 2i\gamma)|^2 \Gamma(n)} \int_0^1s^{2n+1} (P_{(1-s^2)}+Q_{1-s^2})(f\ast_3 \psi)\ast \mu_{s}(z,t)\,ds, 
$$
where we have used that  $ (1-s^2)\frac d{ds}p_{(1-s^2)}= - s p_{(1-s^2)} +  sq_{(1-s^2)}$. Reasoning in the same way as above, we get
$$
|I_3|\le C_1(\gamma)M_{\operatorname{HL}}^0f\ast K(z,t),
$$
and by H\"older 
$$
\|I_3\|_{\infty}\le C_1(\gamma)\|M_{\operatorname{HL}}^0 f\|_{1+\delta}\le C_1(\gamma)\|f\|_{1+\delta}.
$$
Finally $$\|\mathcal{T}_{1+i\gamma}f\|_{\infty} \le \|I_1+I_2\|_{\infty}+\| I_3\|_{\infty} \le C_1(\gamma) \|f\|_{1+\delta} .$$

\end{proof}

With an argument analogous to Proposition \ref{prop:L2}, it can be shown that $\mathcal{T}_{\beta}$ is bounded on $L^2(\H^n)$ for some $\beta<0$. 
\begin{prop}
\label{prop:L2F}
Assume that $n\ge2$ and $\beta\ge -\frac{n}{4}+\frac{5}{12}$. Then for any $\gamma\in \R$,
$$
\|\mathcal{T}_{\beta+i\gamma}f\|_{2}\le C_2(\gamma)\|f\|_{2}.$$
\end{prop}
\begin{proof}
We have to check that 
$$
 (1+\lambda^2)^{-\beta}\Big|\frac{|\lambda|}{2}\psi_{k}^{2\beta+i\gamma+n-1}(\sqrt{|\lambda|})+\frac{k|\lambda|}{n}\psi_{k-1}^{2\beta+i\gamma+n}(\sqrt{|\lambda|})\Big|\le C_2(\gamma)
$$
where $C_2(\gamma)$ is independent of $k$ and $\lambda$. When $\gamma=0$, it follows from the estimates of Lemma~\ref{lem:uniform} (with $\alpha=0,1$) that
\begin{align*}
 &(1+\lambda^2)^{-\beta}\Big|\frac{|\lambda|}{2}\psi_{k}^{2\beta+n-1}(\sqrt{|\lambda|})+\frac{k|\lambda|}{n}\psi_{k-1}^{2\beta+n}(\sqrt{|\lambda|})\Big|\\
 &\le |\lambda|^{-2\beta}\big(|\lambda|^{-2\beta-(n-1)+\frac23}+|\lambda|^{-2\beta-n+\frac53}\big)\\
 &\le C|\lambda|^{-4\beta-n+\frac53},
\end{align*}
for $|\lambda|\ge 1$, which is bounded for $\beta\ge -\frac{n}{4}+\frac{5}{12}$. For $\gamma\neq 0$ we can express $\psi_{k}^{2\beta+i\gamma+n-1}(\sqrt{|\lambda|})$ in terms of $\psi_{k}^{2\beta-\varepsilon+n-1}(\sqrt{|\lambda|})$ for small enough $\varepsilon>0$ and obtain the same estimate. The proof is complete.
\end{proof}

Let us consider the following holomorphic function $\alpha(z)$ on the strip $\{z:0\le \Re z\le 1\}$, given by  $\alpha(z)=\big(\frac{n}{4}-\frac 5{12}-\varepsilon \big)(z-1)+z$ for a small $\varepsilon>0$. We have $\alpha(0)=-\frac{n}{4}+\frac 5{12}+\varepsilon$ and $\alpha(1)=1$. Then, $\mathcal{T}_{\alpha(z)}$ is an analytic family of linear operators. In view of Propositions~\ref{prop:LinftyF} and \ref{prop:L2F}, we can apply Stein's interpolation theorem. Letting $z=u+iv$, we have
$$
\alpha(z)=0 \Longleftrightarrow \Big(\frac{n}{4}-\frac 5{12}-\varepsilon\Big)(u-1)+u=0 \Longleftrightarrow u=\frac{3n-5-12\varepsilon}{3n+7+\varepsilon}.
$$
Since $\varepsilon>0$ is arbitrary,
we obtain
$$
\mathcal{T}_{\alpha(u)}:L^{p_u}(\H^n)\to L^{q_u}(\H^n)
$$
where 
$$
\frac{6}{3n+7-12\varepsilon}<\frac{1}{p_u}<\frac{3n+1-12\varepsilon}{3n+7-12\varepsilon},\quad\frac{1}{q_u}=\frac{6}{3n+7-12\varepsilon}.
$$ 
This leads to the following result.

\begin{thm}
\label{prop:L2en}
Assume that $n\ge 2$ and $\varepsilon>0$. Then $ B_1 :L^{p}(\H^n)\to L^{q}(\H^n)$ for any $p,q$ such that 
$$
\frac{6}{3n+7-12\varepsilon}<\frac{1}{p}<\frac{3n+1-12\varepsilon}{3n+7-12\varepsilon},\qquad\frac{1}{q}=\frac{6}{3n+7-12\varepsilon}.
$$ 
\end{thm}

A version of the inequality in Theorem \ref{prop:L2en} when $ B_1 $ is replaced by $B_r$ can be accomplished easily. First we have the scaling lemma below, we omit the details. 
\begin{lem}
\label{lem:dilatF}
For any  $r> 0$  we have $B_rf=\frac 1r \delta_r^{-1}B_1\delta_rf$.
\end{lem}

From Theorem \ref{prop:L2en} and Lemma \ref{lem:dilatF} we can prove
\begin{cor}
\label{cor:scaledBR }
Assume that $n\ge 2$. Then $ \| B_rf \|_{\frac {3n+7}6} \le \frac cr \, r^{(2n+2)(\frac 6{3n+7} -\frac 1p)} \| f \|_{p} $ for any $\frac {3n+7}{3n+1} < p <\frac {3n+7}{6} $.
\end{cor}

We are now in position to prove Theorem \ref{thm:LPQ}.
\begin{proof}[Proof of Theorem \ref{thm:LPQ}] Let us denote by $\mathbf{F}_n$ the triangle with vertices $(0,1)$, $\big(\frac{2n-1}{2n},\frac{1}{2n} \big)$ and $\big(\frac{3n+1}{3n+7},\frac{3n+1}{3n+7}\big)$, and its dual  $\mathbf{F}'_n$, the triangle with vertices  $(0,0)$, $\big(\frac{2n-1}{2n},\frac{2n-1}{2n} \big)$ and $\big(\frac{3n+1}{3n+7},\frac{6}{3n+7}\big)$.

\begin{figure}[t]
\includegraphics[scale=1]{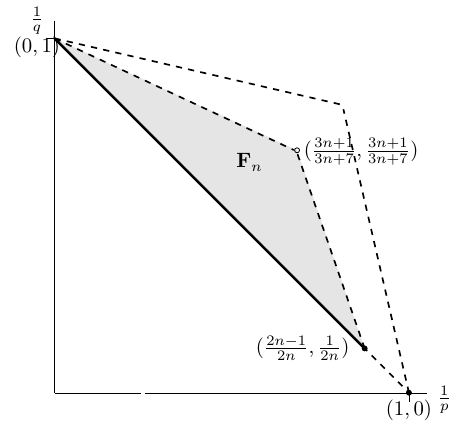}
\includegraphics[scale=1]{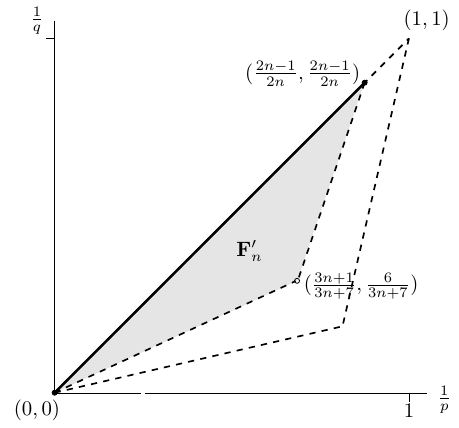}
\caption{Triangle $\mathbf{F}_n'$ shows the region for $L^p-L^q$ estimates for $M_{\delta}$. The dual triangle $\mathbf{F}_n$ is on the left. The outer (dashed) triangles correspond to the  lacunary triangles $\mathbf{S}_n'$ (right) and $\mathbf{S}_n$ (left).}
\label{pic:full}
\end{figure}

The triangle $\mathbf{F}'_n$ is contained in the triangle $\mathbf{S}_n'$, see Figure \ref{pic:full}.  Hence $ \|A_1f  \|_q \le c \| f \|_p $ for all $(\frac 1p, \frac 1q)\in \mathbf{F}'_n $. Using this and Corollary~\ref{cor:scaledBR } we get, in view of \eqref{eq:FTC}, the estimate $\|M_{\delta}f \|_{\frac {3n+7}6} \le     c \| f \|_p $ for $\frac {3n+7}{3n+1} < p <\frac {3n+7}6$. On the other hand,  we also have $ \|M_{\delta}f \|_p \le c\|f \|_p $ for all $p> \frac {2n}{2n-1}$, by \cite{MS} or  \cite{NaT}. After applying Marcinkiewicz interpolation theorem we get the required result.
\end{proof}

\begin{rem}
Actually, a scaling argument allows to state a result for the local maximal function taken over $ \delta^{k+1} \leq t \leq \delta^k$, for any $k\in \Z$.
\end{rem}

\subsection{The continuity property of a local variant of the maximal function}
\label{sec:contF}

We start with the following proposition. 
\begin{prop}
\label{prop:mod}
For all  $  (\tfrac1{p}, \tfrac 1{q}) $  in the interior of the triangle $\mathbf{F}_n'$, there exists $ \nu >0$ so that for all $ 0< \epsilon < \frac 12 $, we have 
\begin{equation}
\bigl\| \sup _{ \substack{ s,r\in [1 ,\delta^{-1}]\\  |  s-r|< \epsilon  }}
 |  A _{r} f - A_s f |  \big\|_ {q} \lesssim \epsilon ^{\nu  }   \| f\| _{p}.   
\end{equation}
\end{prop}
\begin{proof}
It suffices to prove a version of statement at the point $ (\frac12 ,\frac12)$, and then interpolate to the other points in the interior of $\mathbf{F}_n'$. First we have, by Fundamental Theorem of Calculus and H\"older inequality, 
$$ | A_{r} f - A_{s} f | \leq c | r-s | ^{1/2} \Big\|\frac{d}{dr} A_{r}f  \Big\|_{L^{2}(\R)},
$$ 
This gives us
$$
\bigl\|\sup _{ \substack{ s,r\in [1 ,\delta^{-1}]\\ |  s-r|< \epsilon  }}
 | A _{r} f - A_s f |\big\|_2 \leq c \epsilon^{\frac{1}{2}} \|  \partial_{r} A_{r} f \| _{L ^2 (\mathbb H ^{n} \times [1 ,\delta^{-1}])} \lesssim \epsilon^{\frac{1}{2}} \| f\|_2. 
 $$
Moreover, from Theorem  \ref{thm:LPQ},
 $$\| \sup_{r\in [1 ,\delta^{-1}]}  A_{r} f \|_{q_0}   \leq  \| f \|_{p_0} , \quad \text{ for } \quad \Big(\frac 1{p_0}, \frac 1{q_0}\Big)\in  \mathbf{F}_n' $$
 which gives  
\begin{equation}
\bigl\| \sup _{ \substack{ s,r\in [1 ,\delta^{-1}]\\  | s-r|< \epsilon  }}
 |  A _{r} f - A_s f |  \big\|_ {q_0} \lesssim    \|f\| _{p_0},\quad  \text{ for } \quad \Big(\frac 1{p_0}, \frac 1{q_0}\Big)\in  \mathbf{F}_n'.
\end{equation}
 Then the proposition follows immediately by using interpolation. 
\end{proof} 

\begin{thm}
\label{thm:cont}
For all $ (\frac1p, \frac1q) $ in the interior of $\mathbf{F}_n'$, we have for some  $\nu = \nu (n,p,q)>0$,
$$
\|  \sup _{1 \leq r \leq \delta^{-1}} |  A_{r} f -  A_{r} \tau _y f  | 
\|_q 
\lesssim | {y}| ^{\nu}   \| f\|_p, \qquad |  y|<1 .
$$
\end{thm}
\begin{proof}
For $ (\frac 1{p}, \frac1{q}) $ in the interior of the triangle $\mathbf{S}_n'$  we have, by Corollary \ref{cor:dilat2},
\begin{equation}\label{e:littman}
\| A_{r}  -  A_{r} \tau _y \|_ {L ^{p} \mapsto L ^{q}}
\lesssim |y | ^{\nu} , \qquad  |  y| < 1   , 
\end{equation}
for a choice of $ \nu = \nu (n, q,s)>0$. 
The triangle $\mathbf{F}_n'$ is contained in the triangle $\mathbf{S}_n'$. 
Thus, if $ \mathcal T \subset [1,\delta^{-1}]$ is a finite set, it follows from the previous statement that  
\begin{equation} \label{eq: net}
\| \sup _{r\in \mathcal T} | A _{r} f -  A _{r}\tau _y f |\| _ {p_2}
\lesssim  \sharp (\mathcal T) ^{1/p_2}  \cdot |  y | ^{\nu }  \| f\| _{p_1} , \qquad \Big(\frac1{p_1},\frac 1{p_2} \Big) \in \mathbf{F}_n'. 
\end{equation}
Take $ \mathcal T $ to be a $ | y |^{\nu} $-net 
in $ [1 ,\delta^{-1}] $. Then for any $r \in [1 ,\delta^{-1}] $ there exists $ r_{0} \in \mathcal{T}$ such that $|r-r_{0}| < |y|^{\nu}$.
Moreover, by triangular inequality we have
 $$ |A_{r}f - \tau_{y} A_{r} f | \leq 
  |A_{r} f - A_{r_{0}} f | + | A_{r_{0}} f - \tau_{y} A_{r_{0}} f | + | \tau_{y} A_{r_{0}} f - \tau_{y} A_{r} f |  $$
  which gives
$$\| \sup _{r\in [1,\delta^{-1}]} |  A _{r} f - \tau _y A _{r} f | \|_{q} \leq \big\| \sup _{ \substack{ s,r\in [1, \delta^{-1}]\\  | s-r|< \epsilon  }}
 | A _{r} f - A_s f | \big\|_ {q} + \| \sup _{t\in \mathcal T} \| A _{r} f - \tau _y A _{r} f |\| _ {q}.
$$ 
The final result follows  by using  Proposition \ref{prop:mod}  and the inequality \eqref{eq: net}.
\end{proof}

We need a different version of the inequality in Theorem $\ref{thm:cont}$ when the interval $1\le r \le \delta^{-1}$ is replaced by $\delta^{k+1} \le r \le \delta^k $. First we have a slight generalisation of Lemma \ref{lem:dilat}.

\begin{lem}
\label{lem: scal}
Let $r=\ell r'$ for some fixed $\ell$. Then we have $A_rf=\delta_{\ell^{-1}}A_{r'}\delta_{\ell}f $. 
\end{lem}

\begin{proof}
\begin{align*}
A_rf(z,t)&=\int_{|w|=r}f\big(z-w,t-\frac 12 \Im  z\cdot\overline{w}\big)\,d\mu_r(w)\\
&=\int_{|w|=1}f\big(z-rw,t-\frac 12  \Im rz\cdot\overline{w}\big)\,d\mu_1(w)\\
&=\int_{|w|=1}f\big(z-\ell r'w,t-\frac 12  \Im \ell r'z\cdot\overline{w}\big)\,d\mu_1(w)\\
&=\int_{|w|=1}\delta_{\ell}f\big(\ell^{-1}z-r'w,\ell^{-2}t-\frac 12  \Im\ell^{-1}r'z\cdot\overline{w}\big)\,d\mu_1(w)\\
&= \delta_{-\ell}A_{r'}\delta_{\ell}f(z,t).
\end{align*}
\end{proof}
Finally, we can deduce the following.

\begin{cor}
\label{cor:contFull}
Assume that $n\ge2$. For all $ \big(\frac1p, \frac1q\big) $ in the interior of the triangle joining the points $(0,0)$, $\big(\frac{2n-1}{2n},\frac{2n-1}{2n} \big)$ and $\big(\frac{3n+1}{3n+7},\frac{6}{3n+7}\big)$,
we have for some  $\nu  = \nu  (n,p,q)>0$ and $ \big|  \frac{y}{\delta^{k+1}}\big|<1$,
$$
\| \sup _{\delta^{k+1} \leq r \leq \delta^{k}} |  A_{r} f -  A_{r} \tau _y f  |
\|_q 
\lesssim  \big| \frac{y}{\delta^{k+1}}\big|^{\nu }   \delta^{(k+1)(2n+2)(\frac 1q -\frac 1p)}\| f\|_p.
$$
\end{cor}

\begin{proof}

Let $r\in [\delta^{k+1},\delta^k]$. We have $r=\delta^{k+1}r'$ for some $r'\in [1,\delta^{-1}]$. To avoid confusion in the notations, we denote $\ell=\delta^{k+1}$. 
From Lemma \ref{lem: scal} we get
$$ 
A_{r}f- A_{r}\tau_y = \delta_{\ell^{-1}}A_{r'}\delta_{\ell}f - \delta_{\ell^{-1}}A_{r'}\delta_{\ell}(\tau_y f).
$$
Also we observe that $ \delta_{\ell}(\tau_yf)=\tau_{\ell^{-1}y}(\delta_{\ell}f) $. Hence
$$ 
A_{r}f- A_{r} \tau_y = \delta_{\ell^{-1}}A_{r'}\delta_{\ell}f -\delta_{\ell^{-1}}A_{r'}\tau_{\ell^{-1}y}(\delta_{\ell}f).
$$
This gives, for $ \big(\frac1p, \frac1q\big) \in \mathbf{F}_n'$,
\begin{align*}
\| \sup_{\delta^{k+1} \le r\le \delta^{k}} |A_{r}f- A_{r} \tau_y| \|_q &= \ell^{(2n+2)\frac 1q}\| \sup_{1 \le r'\le \delta^{-1}}|A_{r'}\delta_{\ell}f - A_{r'}\tau_{\ell^{-1}y} \delta_{\ell} f| \|_q \\
&\lesssim \ell^{(2n+2)\frac 1q} \big|\frac{y}{\ell} \big|^{\nu } \| \delta_{\ell}f \|_p  \\
& \lesssim  \ell^{(2n+2)(\frac 1q - \frac 1p)}  \big|\frac{y}{\ell} \big|^{\nu } \| f \|_p.
\end{align*}
Above, the first inequality follows from Theorem \ref{thm:cont} and for the second inequality we have used the fact that $\| \delta_{\ell}f \|_p= \ell^{-\frac{2n+2}{p}} \|f \|_p $. The corollary is proved.

\end{proof}

\subsection{Sparse bounds and boundedness properties}
\label{sec:sparseF}

The strategy to get Theorem \ref{thm:sparseF} is the same as in the lacunary case, now making use of Corollary \ref{cor:contFull}. We only provide the details of the main differences. First, a lemma analogous to Lemma \ref{lem:Hn} holds, and the proof is exactly the same.

\begin{lem}
\label{lem:Hn-1}
Let $0<\delta< \frac{1}{96}$. For $Q$ with $\ell(Q)=\delta^k$, $k\in \Z$, we consider
$$
\mathbb{V}_{Q}=\{P\in \mathcal{D}^1_{k+3}: B(z_{P},\delta^{k+1})\subseteq Q\}
$$
and define 
$$
\widetilde M_{Q}f:=\sup_{\delta ^{k+3} \leq r <\delta^{k+2}}A_{r}(f{\bf{1}}_{V_Q})
$$ 
where $ V_{Q}=\cup_{P\in \mathbb{V}_{Q}}P.$ 
Then  
\begin{equation*}
 \sup _{\delta ^{k+3} \leq r <\delta ^{k+2}} A _{r}f 
 \leq  \sum_{\alpha=1} ^{N} \sum_{Q \in \mathcal D_k^{\alpha}}   \widetilde M_Q f.
\end{equation*} 
\end{lem}

It suffices to prove the sparse bound for each of the maximal operators 
\begin{equation} 
M _{\mathcal D^{\alpha}} f := 
\sup _{Q\in \mathcal D^{\alpha}} \widetilde M_Q f , \qquad  1\leq \alpha \leq  N.  
\end{equation}
We fix a grid, and write $ \mathcal D = \mathcal D^\alpha$. With the same linearisation argument as in Section \ref{sec:sparse}, by denoting $\mathcal{D}(Q_0)$ the collection of all dyadic subcubes of $Q_0\in \mathcal{D}$, we define 
$$
E_Q:=\big\{x\in Q : \widetilde M_Qf(x)\ge \frac{1}{2} \sup_{P\in\mathcal{D}(Q_0)} \widetilde M_Pf(x)\big\}
$$ 
for $Q\in\mathcal{D}(Q_0)$. Note that for any $ x \in \H^n $ there exists  a $Q \in\mathcal{D}(Q_0)$ such that  
$$ 
\widetilde M_Qf(x)\ge \frac{1}{2} \sup_{P\in\mathcal{D}(Q_0)} \widetilde M_Pf(x) 
$$  
and hence $ x \in E_Q$.  Now we define $B_Q=E_Q\setminus \cup_{Q'\supseteq Q}E_{Q'}$, so that $\{B_Q: Q\in\mathcal{D}(Q_0)\}$ are disjoint and moreover $\cup_{Q\in\mathcal{D}(Q_0)}B_Q=\cup_{Q\in\mathcal{D}(Q_0)}E_Q$. Then it follows that
$$
\notag\langle\sup_{Q\in\mathcal{D}(Q_0)}\widetilde M_Qf,g\rangle \le 2 \sum_{Q\in\mathcal{D}(Q_0)}\langle \widetilde M_Qf_1,f_2\mathbf{1}_{B_Q}\rangle.
$$
Thus, defining  $(f_2)_Q:=f_2\mathbf{1}_{B_Q}$ we deal with $\sum_{Q\in\mathcal{D}(Q_0)}\langle \widetilde M_Qf_1,(f_2)_Q\rangle$.

\begin{lem}
 Let $1<p,q<\infty$ be such that $\big(\frac1p,\frac1q\big)$ in the interior of the triangle joining the points $(0,1)$, $\big(\frac{2n-1}{2n},\frac{1}{2n} \big)$ and $\big(\frac{3n+1}{3n+7},\frac{3n+1}{3n+7}\big)$. Let $f_1=\mathbf{1}_{F}$ and let $ f_2 $ be any bounded function supported in $ Q_0$. Let $C_0>1$ be a constant and let $\mathcal{Q}$ be a collection of dyadic subcubes of $Q_0\in \mathcal{D}$ for which the following holds
\begin{equation}
\sup_{Q'\in \mathcal{Q}}\sup_{Q: Q'\subset Q\subset Q_0}\frac{\langle f_1\rangle_{Q,p}}{\langle f_1\rangle_{Q_0,p}}<C_0.
\end{equation}
Then there holds 
$$
\sum_{Q\in \mathcal{Q}}\langle \widetilde M_Qf_1,(f_2)_Q\rangle\lesssim |Q_0|\langle f_1\rangle_{Q_0,p}\langle f_2\rangle_{Q_0,q}.
$$
\end{lem}

\begin{proof}

We perform a Calder\'on--Zygmund decomposition of $f_1=g_1+b_1$ at height $2C_0\langle f_1 \rangle_{Q_0,p}$ as in \eqref{eq:decoCZ}, where the bad cubes $\mathcal{B}$ result from the collection of (maximal) dyadic subcubes of $Q_0$  so that
\begin{equation}
\label{eq:stopp}
\langle f_1\rangle_{Q,p}>2C_0\langle f_1 \rangle_{Q_0,p}.
\end{equation}
We aim to bound the bilinear form 
$$
\big|\sum_{Q\in \mathcal{Q}}\langle \widetilde M_Qf_1,(f_2)_Q\rangle\big|
$$
The term carrying the good function $g_1$ is bounded analogously as in Lemma \ref{lem:key}. 
For the term involving $b_1$, we have, for any $Q\in \mathcal{Q}$ with $\ell(Q)=\delta^s $,
\begin{equation*}
\big|\sum_{Q\in \mathcal{Q}}\langle \widetilde M_Qb_1,(f_2)_Q\rangle\big|\le \sum_{k=1}^{\infty}\sum_{Q\in \mathcal{Q}}|\langle \widetilde M_QB_{1,s+k},(f_2)_Q\rangle|.
\end{equation*}
As shown in \cite[Lemma 3.4]{Lacey}, we can replace $\widetilde M_Q \phi(x)$ by 
$$L_{Q}\phi (x) := A_{r_{Q}(x)} \phi(x) \mathbf{1}_{V_Q}(x) 
$$ 
where $r_{Q} : Q \mapsto [\delta ^{s+3}, \delta^{s+2}]$ is a measurable function. 

By making use of  the mean zero property of $b_1$, we see that
\begin{align*}
&|\langle L_QB_{1,s+j},(f_2)_Q\rangle|=|\langle B_{1,s+j},L_Q^*(f_2)_Q\rangle|\\
&= \sum_{P\in B(s+j)}\big|\int_P L_Q^*(f_2)_Q(x)B_{1,s+j}(x)\,dx\big|\\
&\le  \sum_{P\in B(s+j)}\frac{1}{|P|}\Big|\int_P\int_P\big[L_Q^*(f_2)_Q(x)-L_Q^*(f_2)_Q(x')\big]B_{1,s+j}(x)\,dx\,dx'\Big|\\
&\le \sum_{P\in B(s+j)}\frac{1}{|P|}\Big|\int_{P^{-1}P}\int_P\big[L_Q^*(f_2)_Q(x)-\tau_yL_Q^*(f_2)_Q(x)\big]B_{1,s+j}(x)\,dx\,dy\Big|\\
&\lesssim \frac{1}{|P_0|}\int_{P_0}\Big|\int_Q(f_2)_Q(x)(L_Q-L_Q\tau_{-y})B_{1,s+j}(x)\,dx\Big|\,dy,
\end{align*}
where $P_0=B(0,\delta^{s+j-1})$ and we used that $P^{-1}x\subset P^{-1}P\subset P_0$. 
Now 
$$
|L_{Q}f(x)-L_{Q}\tau_{-y}f(x)| \leq   \sup _{\delta^{s+3}\leq r \leq \delta^{s+2}}|  A_r f(x) -  A_r\tau _{-y} f(x) |.  
$$
Then, under the assumptions of $(1/p,1/q)$ in Corollary \ref{cor:contFull}, it follows
$$
\| L_{Q}f-L_{Q}\tau_{-y}f\|_{q'} \leq \Big|\frac{y}{\delta^{s+3}}\Big|^{\nu}\delta^{(j+2)(n+2)(\frac{1}{q'} - \frac1p)}\| f \|_p \quad \text{ for all } \Big(\frac{1}{p}, \frac{1}{q'}\Big) \in F'_n.
$$
So for all $(\frac{1}{p}, \frac{1}{q}) \in F_n$,
\begin{align*}
&|\langle L_QB_{1,q-j},(f_2)_Q\rangle|\\
&\lesssim \frac{1}{|P_0|}\int_{P_0} \Big|\frac{y}{\delta^{s+3}}\Big|^{\nu}\delta^{(j+2)(n+2)(\frac 1{q'} - \frac1p)}\| B_{1,q-j} 1_Q \|_p \|(f_2)_Q \|_q\,dy \\
&=\delta^{j\nu}|Q|\langle B_{1,q-j}\mathbf{1}_Q\rangle_{Q,p}\langle (f_2)_Q\rangle_{Q,q}.
\end{align*} 
Finally, recall the inequality \eqref{eq:claimII}
\begin{equation*}
\sum_{Q\in \mathcal{Q}}|Q|\langle B_{1,s+j}\mathbf{1}_Q\rangle_{Q,p}\langle f_2\mathbf{1}_{B_Q}\rangle_{Q,q}\lesssim |Q_0|\langle f_1\rangle_{Q_0,p}\langle f_2\rangle_{Q_0,q},
\end{equation*}
for all $j\ge1$ and for all $1<p,q<\infty$ such that $\big(\frac1p,\frac1q\big)$ are in the interior of the triangle joining the points $(0,1), (1,0)$ and $(1,1)$, including the segment joining $(0,1)$ and $(1,0)$, excluding the endpoints (observe that this triangle contains the triangle given by the assumptions of the lemma). 
This concludes the proof of the lemma.
\end{proof}

The proof of Theorem \ref{thm:sparseF} now follows the same steps of the proof of Theorem \ref{thm:mainSH} with $A_Qf$ replaced by $\widetilde{M}_Qf$. We omit the details.

From the sparse domination results, analogously as in the lacunary case, a number of weighted estimates can be immediately deduced.
\begin{cor}
\label{cor:weightF}
Let $n\ge2$ and define
$$
\frac{1}{\phi(1/p_0)}=\begin{cases}1-\frac{1}{p_0}\frac{6}{3n+1}, \quad 0<\frac1p_0\le \frac{3n+1}{3n+7},\\
\frac{6n^2-n-7}{9n-7}\Big(\frac{2n-1}{2n}-\frac1p_0\Big)+\frac{1}{2n}, \quad \frac{3n+1}{3n+7}<\frac1p_0<\frac{2n-1}{2n}.
\end{cases}
$$
Then $M_{\operatorname{lac}}$ is bounded on $L^p(w)$ for $w\in A_{p/p_0}\cap \operatorname{RH}_{(\phi(1/p_0)'/p)'}$ and all $1<p_0<p<(\phi(1/p_0))'$.
\end{cor}

\begin{center}{\bf Acknowledgments}
\end{center}
We are grateful to the anonymous referee for his/her very careful reading of the original manuscript and constructive comments that have contributed to improve the presentation of the paper.

This work was mainly carried out when the first, second and the fourth author were visiting the third author in  Bilbao. They wish to thank BCAM in general and Luz Roncal in particular for the warm hospitality they enjoyed during their visit. The last author fondly remembers the daily shots of cortado as well as the changing colours of The Puppy! 

All the four authors are thankful to Michael Lacey for answering  several queries and offering clarifications regarding his work \cite{Lacey}. They would also like to thank Kangwei Li for fruitful discussions and David Beltran for a careful reading of the manuscript and his helpful suggestions.

The first, second, and third-named authors were supported by 2017 Leonardo grant for Researchers and Cultural Creators, BBVA Foundation. The first-named author was also supported by Inspire Faculty Fellowship [DST/INSPIRE/04/2016/000776]. The fourth author was visiting Basque Center for Applied Mathematics through a Visiting Fellow programme.
The third and fourth-named author were also supported 
by the Basque Government through
BERC 2018--2021 program, by Spanish Ministry of Science, Innovation and Universities through BCAM Severo Ochoa accreditation SEV-2017-2018 and the project MTM2017-82160-C2-1-P funded by AEI/FEDER, UE. The third author also acknowledges the RyC project RYC2018-025477-I and IKERBASQUE.

\end{document}